\documentclass{amsart}

\usepackage{a4wide}
\usepackage{lineno}
\usepackage{mathtools,amsmath,amsthm,amssymb,amsfonts}
\usepackage{graphicx,color,xcolor,colortbl}
\usepackage{subfigure,caption,subcaption}
\usepackage{tabularx,array}
\usepackage[colorlinks=true,allcolors=blue]{hyperref}
\usepackage{url}
\usepackage{tikz}
\usepackage{enumitem}
\usepackage[normalem]{ulem}

\newtheorem{theorem}{Theorem}
\newtheorem{lemma}{Lemma}
\newtheorem{corollary}[lemma]{Corollary}

\newcommand{\RR}{\mathbb{R}}

\newcommand{\NN}{\mathbb{N}}
\newcommand{\prim}[1]{q_{#1}^{\#}}
\newcommand{\lnthree}{\ln \ln \ln}

\DeclareMathOperator{\rad}{rad}

\begin{document}

\title[Finite arithmetic form of  Robin's inequality and equivalence to  RH]{A finite arithmetic form of Robin's inequality and its equivalence to the Riemann hypothesis}

\author{Challenger Mishra}
\address{Department of Computer Science and Technology, University of Cambridge CB3 0FD}
\curraddr{}
\email{cm2099@cam.ac.uk}
\thanks{}

\author{Rahul Sarkar}
\address{Department of Mathematics, University of California, Berkeley, CA 94720}
\curraddr{}
\email{rsarkar@berkeley.edu}
\thanks{}


\date{}
\dedicatory{}
\commby{}
\begin{abstract}
Robin's inequality reformulates the Riemann hypothesis as a bound on the sum-of-divisors function $\sigma$. We introduce a finite-truncation version of Robin's inequality determined intrinsically by the arithmetic structure of $n$, 
\[
\frac{\sigma(n)}{n}
<
e^\gamma
\sum_{j=0}^{\omega(n)}
\frac{(\ln\ln\ln n)^j}{j!},
\qquad n>5040,
\]
where $\omega(n)$ denotes the number of distinct prime factors of $n$, and $\gamma$ is the Euler-Mascheroni constant. The resulting bound is pointwise stronger than Robin's inequality and we prove it unconditionally for integers with $\omega(n)\le 6$, primorials, odd integers and square-free integers. We also prove that this inequality is equivalent to Robin's inequality, and hence to the Riemann hypothesis. Consequently, the Riemann hypothesis is equivalent to the truncated inequality holding for all colossally abundant numbers. We further show that if the inequality fails, its minimal counterexample must be a superabundant number, placing any obstruction within a rigid extremal class of integers. 
\end{abstract}

\maketitle

\section{Introduction}
\label{sec:intro}

In \cite{robin1984grandes}, Robin conjectured the inequality (called \textit{Robin's inequality}) $\sigma(n) < e^{\gamma} n \ln \ln n$, for all $n > 5040$, where $\sigma(n)$ is the sum-of-divisors function of $n$, and $\gamma=0.577215...$ is the Euler-Mascheroni constant, and proved that it is equivalent to the Riemann hypothesis (RH). Earlier it was shown by Ramanujan that RH itself implies this inequality~\cite{ramanujan1915highly}, albeit the asymptotic form, and hence this inequality is also called the \textit{Ramanujan-Robin criterion} \cite[Chapter~7]{broughan2017equivalents}. Since then, this inequality has been unconditionally proven for many classes of integers \cite{luca2025robin}. These include: odd integers greater than~$9$ and square--free integers greater than~$30$ \cite{choie2007robin}; 
integers greater than~$720$ expressible as a sum of two squares \cite{banks2009nicolas}; 
integers free of $p^5$ \cite{choie2007robin}, $p^7$ \cite{sole2012robin}, $p^{11}$ \cite{broughan2015robin}, 
$p^{20}$ \cite{platt2021robin}, and $p^{21}$ \cite{axler2023robin} divisors for any prime~$p$.  

If Robin's inequality is true, then for every $n > 5040$ there exists an integer $A(n)$ such that
\begin{equation}
\label{eq:existence-An}
    \sigma(n) < e^{\gamma} n \sum_{j=0}^{A(n)} \frac{(\ln \ln \ln n)^j}{j!}.
\end{equation}
It is easy to see how the right hand side above reduces to $e^\gamma n \ln \ln n$ when summed to infinity instead, recovering Robin's inequality. In this paper we consider two cases: $A(n)=\pi(n)$ and $A(n)=\omega(n)$, where $\pi$ is the \textit{prime counting function} that counts the number of primes less than or equal to  $n$,  and $\omega$ is the \textit{prime omega function}, the number of distinct primes in the  factorization of $n$. We refer to the inequality \eqref{eq:existence-An} for these two cases as the \textit{$\pi-$inequality} and \textit{$\omega-$inequality} respectively. For both these cases, we prove the inequality~\eqref{eq:existence-An} for many classes of integers. Specifically we establish conditions under which the inequality holds unconditionally, and find properties of a hypothetical minimal counterexample in each case. For the $\pi-$inequality we show that the minimal counterexample if it exists is a superabundant number, whereas for the $\omega-$inequality, we show that the minimal counterexample satisfies several rigid multiplicity bounds and is in fact a superabundant number. We now proceed with the \textit{$\omega-$inequality}:
\begin{equation}
\label{eq:omega-n-conj}
    \frac{\sigma(n)}{n} < C(n) := e^{\gamma} \sum_{j=0}^{\omega(n)} \frac{(\ln \ln \ln n)^j}{j!}.
\end{equation}
\begin{figure}[t]
\centering 
\includegraphics[width=0.95\textwidth]{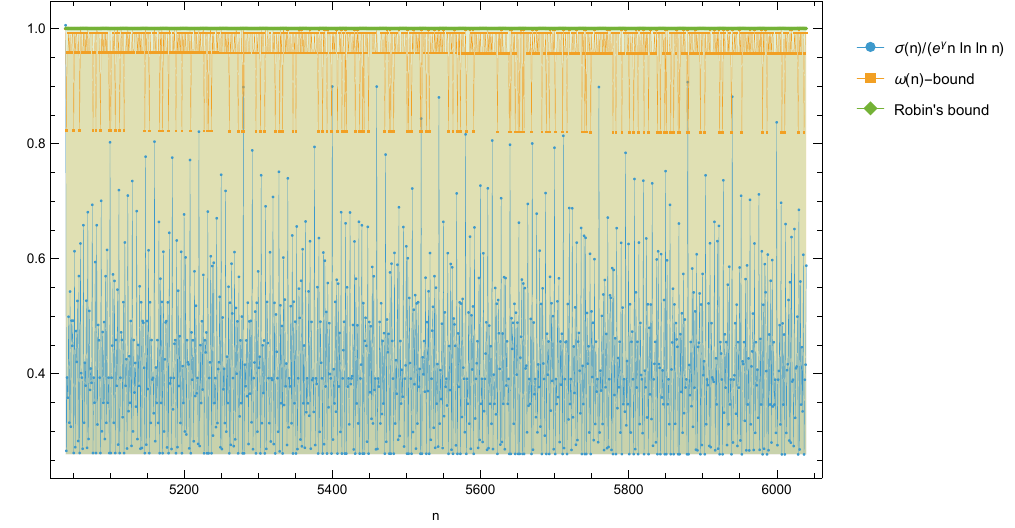}
\caption{We plot the ratio $\frac{\sigma(n)}{e^\gamma n \ln \ln n}$ in the range $n\in [5040,6040]$ together with two upper bounds: (1) the $\omega-$inequality: $\frac{1}{\ln\ln n}\sum_{j=0}^{\omega(n)} \frac{(\ln \ln \ln n)^j}{j!}$, and (2) Robin's bound which in this normalization is identically 1. The point \(n=5040\), the largest known counterexample common to both inequalities, is included for reference.}
\label{fig:evidence_plot}
\end{figure}
For integers in the range \(3 \le n \le 5040\), Robin’s inequality fails for precisely the values~\cite{OEIS_A067698}:
\begin{align*}
&3, 4, 5, 6, 8, 9, 10, 12, 16, 18, 20, 24, 30, 36, 48, 60,&\\
&72, 84, 120, 180, 240, 360, 720, 840, 2520, 5040.&
\end{align*}
Notably, the \(\omega-\)inequality~\eqref{eq:omega-n-conj} fails for exactly the same set of integers, with the sole additional exception \(n=7\).
In Figure~\ref{fig:evidence_plot}, we present numerical evidence for both Robin’s inequality and the \(\omega-\)inequality for integers \(n \ge 5040\), where the largest numerically observed counterexample which is common to both inequalities and occurs at \(n=5040=7!\) is kept deliberately. In this figure we have normalized the $\omega-$inequality~\eqref{eq:omega-n-conj} by $\ln \ln n$ for better visualization. In this normalization, Robin's bound is identically $1$. 

\subsection{Organization of the paper}
\label{ssec:structure}

Section~1 introduces the $\omega-$ and $\pi-$inequalities, recalls the
relevant background on Robin's inequality, superabundant and colossally
abundant numbers, fixes the notation used throughout the paper, and
states the main results. Section~\ref{sec:K-less-than-6} proves the
$\omega-$inequality for all $n>5040$ with $\omega(n)\leq 6$.
Section~\ref{sec:primorials} establishes the $\omega-$inequality for
primorial numbers and deduces the square-free and odd cases.
Section~\ref{sec:minimal-counterex} studies hypothetical minimal
counterexamples, deriving strong restrictions on their prime
factorizations and multiplicities and ultimately proving that a minimal
counterexample to the $\omega-$inequality is superabundant.
Section~\ref{sec:equivalance} proves that the $\omega-$inequality is
equivalent to Robin's inequality, and hence to the Riemann hypothesis.
Appendix~A collects auxiliary analytic estimates and additional
multiplicity bounds, while Appendix~B records the finite computations
used in the proofs and the corresponding reproducibility information.

\subsection{Background and notation}
\label{ssec:background}

We denote the set of all positive integers as $\NN$. For $n \in \NN$, we denote its prime factorization as $n=p_1^{a_1} p_2^{a_2} \dots  p_K^{a_K}$, where $p_1 < p_2 < \dots < p_K$ are primes, and $a_k \in \NN$ for all $k$. The \textit{prime support} of $n$ is the set of all prime numbers appearing in its prime factorization. We define the \textit{radical} of $n$, denoted $\rad(n)$, as the product $\prod_{k=1}^K p_k$. Integers such that $n=\rad(n)$ are called \textit{square-free} integers. The sum-of-divisors function $\sigma(n)$ is the sum of all positive integers that divide $n$, and it has the property that if $m, n \in \NN$ are coprime, then $\sigma(mn) = \sigma(m) \sigma(n)$. Using this fact,  $\sigma(n)$ and $\sigma(n)/n$ can be expressed as
\begin{equation}
\label{eq:sigma-n}
\begin{split}
    \sigma(n) &= \prod_{k=1}^{K} \left( 1 + p_k + \dots + p_k^{a_k} \right) = \prod_{k=1}^{K} \frac{p_k^{a_k+1}-1}{p_k-1}, \\
    \frac{\sigma(n)}{n} &= \prod_{k=1}^{K} \left( 1 + \frac{1}{p_k} + \dots + \frac{1}{p_k^{a_k}} \right) = \prod_{k=1}^{K} \frac{1 - (1/p_k)^{a_k+1}}{1-1/p_k}.
\end{split}
\end{equation}
We refer to the quantity $\sigma(n)/n$ as the \textit{divisor ratio}. The number of distinct prime factors of $n$ is denoted by
$\omega(n)$, so that $K=\omega(n)$. Throughout this paper, we will denote the $k^{\text{th}}$ prime number as $q_k$, for all $k \in \NN$, that is, $q_1 = 2$, $q_2 = 3$, $q_3 = 5$, and so on. We also define the \textit{primorial numbers}
\begin{equation}
\label{eq:primorial}
    \prim{k} := \prod_{j=1}^{k} q_j, \quad k \in \NN.
\end{equation}
Expressions of the form $\sum_{p \le x}$ (resp. $\prod_{p \le x}$) will always denote sums (resp. products) over all primes $p$ less than or equal to $x \in \RR$, unless specified otherwise. We denote $\vartheta(x) := \sum_{p \le x} \ln p$, and note that $\vartheta(q_k) = \ln \prim{k}$ for all $k \in \NN$.

Recall that a positive integer \(n\) is superabundant (SA) if the divisor ratio satisfies 
\begin{equation*}
    \frac{\sigma(m)}{m}< \frac{\sigma(n)}{n},~\text{for all~} 0<m<n.
\end{equation*}
Superabundant numbers were studied in great detail by Alaoglu and Erd\"os in~\cite{alaoglu1944highly}, who showed that if $n$ is superabundant with prime factorization $n=p_1^{k_1} p_2^{k_2}\ldots p_{\omega(n)}^{k_{\omega(n)}}$, then its factorization satisfies several properties, of which three are particularly significant, and these will interest us the most in this paper:
\begin{itemize}
    \item $\rad(n)=q^\#_{\omega(n)}$,
    \item $k_1\ge k_2\ge \ldots \ge k_{\omega(n)}$ (\cite[Theorem 1]{alaoglu1944highly}),  
    \item $k_{\omega(n)}=1$ except when $n=4$ or $36$ (\cite[Theorem 3]{alaoglu1944highly}). 
\end{itemize}
The first property above says that the prime factorization of $n$ contains the first $\omega(n)$ primes. It was later shown that the minimal counterexample to Robin's inequality, if it exists, must be a superabundant number~\cite{akbary2009superabundant}. We will return to this point in Section~\ref{sec:minimal-counterex}. 

In the context of Robin's inequality, another class of integers namely \textit{colossally abundant} numbers (CA) is well-studied \cite[Chapter~6]{broughan2017equivalents}. A positive integer $n$ is called a colossally abundant number if $\exists~\epsilon>0$, dependent on $n$, for which, 
\begin{equation*}
    \frac{\sigma(m)}{m^{1+\epsilon}} \le \frac{\sigma(n)}{n^{1+\epsilon}},~\text{ for all } m > 1.
\end{equation*}
Every colossally abundant number is also a superabundant number. It was shown by Robin~\cite[Proposition 1]{robin1984grandes} that if Robin's inequality holds for two consecutive colossally abundant numbers, then it holds for all integers in between. 

We define the \textit{truncated exponential sum} function $P_N(x):= \sum_{j=0}^{N}\frac{x^j}{j!}$, for any $ x \in \mathbb{R}$, and integer $N \ge 0$. Note that for all $N \in \mathbb{N}$, the function $P_N$ is positive and strictly increasing for all $x \ge 0$. Moreover, differentiating in $x$ gives $P'_N(x) = P_{N-1}(x)$.  Also, for $x \ge 0$, and for all $N \ge 0$, we define the \textit{truncated exponential tail function} $T_N(x) := e^x - P_N(x)$, and we note that it is a strictly increasing function, for $x \ge 0$, for every $N \ge 1$. The \textit{Lambert $W$ function}, denoted $W$, inverts the map $[0,\infty) \ni x \mapsto x e^x$, and we will take its domain to be $[0,\infty) \subset \RR$. Since $x e^x$ is increasing and injective, so is $W$, and satisfies the simple bound $W(x) \le \frac{1}{2} \left( \sqrt{4x+1} -1 \right)$, for all $x \ge 0$. Moreover, it satisfies the identity $e^{W(x)} = \frac{x}{W(x)}$, for all $x \ge 0$.

\subsection{Main results}
\label{ssec:contrib}

The main results of this paper are as follows.
\begin{enumerate}[label=(\roman*)]
    \item \emph{Unconditional results.}
    We prove that the $\omega-$inequality~\eqref{eq:omega-n-conj}
    holds unconditionally for all $n>5040$ with $\omega(n)\leq 6$,
    and for all primorial numbers $\prim{k}>5040$. As consequences,
    the $\omega-$inequality also holds for all square-free integers and
    all odd integers greater than $5040$.

    \item \emph{Structure of a minimal counterexample.}
    We show that a hypothetical minimal counterexample $n$ to the
    $\omega-$inequality, with $\omega(n)=K$, has prime support
    \[
        \rad(n)=q_K^\#,
    \]
    has non-increasing prime multiplicities, and has unit multiplicity
    at its largest prime factor. We further obtain quantitative bounds
    on its multiplicities, reducing the search at fixed $K$ to finitely
    many exponent tuples, and prove that any minimal counterexample must
    be superabundant. We also establish the corresponding superabundance result for a minimal counterexample to the $\pi-$inequality.

    \item \emph{Equivalence with the Riemann hypothesis.}
    We prove that the $\omega-$inequality is equivalent to Robin's
    inequality, and hence to the Riemann hypothesis. 
    We give a
    proof of this equivalence using the structure of a
    minimal counterexample. 
    As a consequence, the Riemann hypothesis is also
    equivalent to the $\omega-$inequality restricted to colossally
    abundant numbers.
\end{enumerate}

\section{\texorpdfstring{The $\omega-$inequality for $\omega(n) \le 6$}{}}
\label{sec:K-less-than-6}

The goal of this section is to prove that the $\omega-$inequality holds
for all $n>5040$ with $\omega(n)\leq 6$, as stated in
Lemma~\ref{lem:six-primes}. We first establish two auxiliary lemmas that
will be used both in its proof and repeatedly throughout the paper.

\begin{lemma}[Monotonicity]
\label{lem:monotonicity}
Let $m, n >e^e$ be integers with $m < n$ and $\omega(m) \le \omega(n)$. Then $C(m) < C(n)$.
\end{lemma}
\begin{proof}
Proof is trivial.
\end{proof}
Lemma~\ref{lem:monotonicity} immediately yields the following
substitution principle.
\begin{lemma}[Substitution]
\label{lem:substitution}
Let $n = q_1^{a_1} \dots q_K^{a_K}$ for some $a_1, \dots, a_K \in \NN$, such that $\frac{\sigma(n)}{n} < C(n)$ holds. Then for any $m$ with the prime factorization $m=p_1^{a_1} \dots p_K^{a_K}$, we have 
\begin{equation}
\label{eq:substitution}
\frac{\sigma(m)}{m} \le \frac{\sigma(n)}{n} < C(n) \le C(m),   
\end{equation}
where the first and last inequalities are equalities if and only if $m=n$.
\end{lemma}
\begin{proof}
Assume $m \ne n$. By definition, $q_k$ is the $k^{\text{th}}$ prime, and thus $q_k \le p_k$, for every $k$, and there exists $j$ such that $q_j < p_j$. This implies $n < m$, and the inequality $C(n) < C(m)$ now follows from Lemma~\ref{lem:monotonicity}, since $\omega(m) = \omega(n) = K$. For the first inequality in~\eqref{eq:substitution}, note that $q_k \le p_k$ implies $ \left( 1 + \frac{1}{p_k} + \dots + \frac{1}{p_k^{a_k}} \right) \le \left( 1 + \frac{1}{q_k} + \dots + \frac{1}{q_k^{a_k}} \right) $ for every $k$, and the inequality is strict when $k=j$. Together with \eqref{eq:sigma-n}, this implies $\frac{\sigma(m)}{m} < \frac{\sigma(n)}{n}$.
\end{proof}

The substitution Lemma~\ref{lem:substitution} reduces a fixed exponent tuple to the case
supported on the first $K$ primes. We first apply this reduction when
$K\leq 3$.

\begin{lemma}[Three primes]
\label{lem:three-primes}
Let $m = p_1^{a_1} \dots p_K^{a_K}$ be the prime factorization of an arbitrary positive integer $m$, with $\omega(m)=K$. Then $\frac{\sigma(m)}{m} < \prod_{k=1}^{K} \frac{q_k}{q_k - 1}$, and moreover, when $K \le 3$, and $m > 5040$, we have $\frac{\sigma(m)}{m} < C(m)$.
\end{lemma}

\begin{proof}
Consider the number $n := q_1^{a_1} \dots q_K^{a_K}$. Then by the same argument as in the proof of Lemma~\ref{lem:substitution}, we have 
\begin{equation}
\label{eq:three-primes-proof-1}
    \frac{\sigma(m)}{m} \le \frac{\sigma(n)}{n} = \prod_{k=1}^{K} \left( 1 + \frac{1}{q_k} + \dots + \frac{1}{q_k^{a_k}} \right) < \prod_{k=1}^{K} \frac{q_k}{q_k - 1} =: r_K.
\end{equation}
Next, assuming $m > 5040$, we have $\lnthree(m) > 0.7622$. Therefore, $C(m) > e^\gamma \sum_{k=0}^{K} \frac{0.7622^k}{k!}$. Additionally assuming $K \le 3$, one can check numerically that $r_K < e^\gamma \sum_{k=0}^{K} \frac{0.7622^k}{k!}$ (see \ref{list:power-three-primes}), which completes the proof\footnote{This simple proof breaks down for $K=4$ and higher integer values of $K$.}.
\end{proof}

The argument of Lemma~\ref{lem:three-primes} extends to larger values of $\omega(n)$,
at the cost of a finite computational verification. We carry this out
for $\omega(n)\leq 6$.

\begin{lemma}[Six primes]
\label{lem:six-primes}
Let $m \in \NN$ be such that $\omega(m) \le 6$. If $m > 5040$, then $\frac{\sigma(m)}{m} < C(m)$.
\end{lemma}
\begin{proof}
Let $K = \omega(m)$, and assume $4\le K \le 6$, as the $K \le 3$ case is already proved in Lemma~\ref{lem:three-primes}. The proof proceeds similarly to Lemma~\ref{lem:three-primes}, except we first take $m > 1.5 \times 10^8$, so that $\lnthree(m) > 1.076$. Then it follows that $C(m) > e^{\gamma} \sum_{k=0}^K \frac{1.076^k}{k!}$. With $r_K$ as defined in \eqref{eq:three-primes-proof-1}, a numerical check shows $r_K < e^{\gamma} \sum_{k=0}^K \frac{1.076^k}{k!}$, for $K = 4, 5, 6$ (see \ref{list:power-six-primes}). Finally, the case $m\in [5041, 1.5 \times 10^8]$ is also verified to be true by another numerical check (see \ref{list:check-omega-n-1b}). 
\end{proof}
The contrapositive of Lemma~\ref{lem:substitution} also gives the following observation:
\begin{lemma}
\label{lem:violation}
Let $m = p_1^{a_1} \dots p_K^{a_K}$ be the prime factorization of a positive integer $m$, such that $\frac{\sigma(m)}{m} \ge C(m)$. Then, for $n := q_1^{a_1} \dots q_K^{a_K}$, $\frac{\sigma(n)}{n} \ge C(n)$.
\end{lemma}
\section{\texorpdfstring{$\omega-$}{}inequality for primorial numbers and some consequences}
\label{sec:primorials}
Robin's inequality is known unconditionally for primorial numbers
greater than $5040$~\cite[Theorem~1.1]{choie2007robin}. We prove the stronger $\omega-$inequality for
the same class and derive the square-free and odd cases as consequences.
We use the following estimates of Rosser--Schoenfeld~\cite{rosser1962approximate} and Bertrand's postulate~\cite{ramanujan1919proof}, which will be used repeatedly in the rest of the paper.

\begin{lemma}[Rosser-Schoenfeld]
\label{lem:rosser-schoenfeld}
The following bounds on arithmetic functions are true:
\begin{enumerate}[label=(\alph*)]
    \item \label{rs:1} (Eq.~3.20) $\sum_{p \le x} 1/p < \ln \ln x + B + 1/ (\ln x )^2$, for $x > 1$, where $B = 0.26149...$.
    \item \label{rs:2} (Eq.~3.16) $x (1 - 1/\ln x) < \vartheta(x)$, for $x \ge 41$.
    \item \label{rs:2-1} (Eq.~3.14) $x (1 - 1/(2\ln x)) < \vartheta(x)$, for $x \ge 563$.
    \item \label{rs:3} (Eq.~3.15) $\vartheta(x) < x (1 + 1/(2 \ln x))$, for $x > 1$.
    \item \label{rs:4} (Eq.~3.13) $q_k < k ( \ln k + \ln \ln k)$, for $\NN \ni k \ge 6$.
    \item \label{rs:5} (Eq.~3.29) $\prod_{p \le x} \frac{p}{p-1} < e^\gamma \ln x \left( 1 + \frac{1}{2(\ln x)^2} \right)$, for $x \ge 286$.
    \item \label{rs:6} (Eq.~3.5) $\frac{x}{\ln x} < \pi(x)$, for $x \ge 17$.
\end{enumerate}
\end{lemma}

\begin{lemma}[Bertrand's postulate~\cite{ramanujan1919proof}]
\label{lem:bertrand}
If $p$ is a prime, then there exists another prime $p'$ such that $p < p' < 2p$. In particular, $q_{k+1} < 2 q_k$, for all $k \in \NN$.
\end{lemma}

\subsection{The primorial numbers}
\label{ssec:primorial-proof}

We begin with an auxiliary estimate for primorials.

\begin{lemma}
\label{lem:primorial-helper}
Let $K, \ell \in \NN$. Then, we have the following:
\begin{enumerate}[label={(\alph*)}]
    \item \label{ph:1} $\ln(q_K) \left( 1 -  \frac{\alpha}{(\ln(q_\ell))^2} \right) < \ln \vartheta(q_K)$, for all $K \geq \ell \geq 13$, where $\alpha = \left(\frac{\ln(q_\ell)}{\ln(q_\ell) - 1} \right)$.
    \item \label{ph:2} For all $K \in \NN$, with $B$ defined as in Lemma~\ref{lem:rosser-schoenfeld}\ref{rs:1}, it holds that
    \begin{equation*}
        \frac{\sigma(\prim{K})}{\prim{K}} < e^B \ln(q_K) \exp \left( \frac{1}{(\ln(q_K))^2} \right).
    \end{equation*}
    \item \label{ph:3} For all $K \ge 3$, we have
    \begin{equation*}
        C(\prim{K}) > e^{\gamma} \ln \vartheta(q_K) \left( 1 - \frac{(\ln \ln \vartheta(q_K))^{K+1}}{(K+1)!} \right).
    \end{equation*}
\end{enumerate}
\end{lemma}

\begin{proof}
(a) Since $K \ge 13$, we have $q_K \ge 41$, and so we have $\vartheta(q_K) > q_K (1 - 1/\ln(q_K))$, by Lemma~\ref{lem:rosser-schoenfeld}\ref{rs:2}. Taking log and using Lemma~\ref{lem:log-bounds}\ref{lg:1}, we get
\begin{equation*}
\begin{split}
\ln \vartheta(q_K) &> \ln(q_K) + \ln(1 - 1/\ln(q_K)) > \ln(q_K) - \frac{\alpha}{\ln(q_K)} \\
&= \ln(q_K) \left(1 - \frac{\alpha}{(\ln(q_K))^2} \right) > \ln(q_K) \left(1 - \frac{\alpha}{(\ln(q_\ell))^2} \right).
\end{split}
\end{equation*}

(b) From the expression of $\prim{K}$ in \eqref{eq:primorial}, the divisor ratio in \eqref{eq:sigma-n}, and the fact that $\ln(1+x) < x$, for all $x > 0$ (Lemma~\ref{lem:log-bounds}\ref{lg:2}), we have $\frac{\sigma(\prim{K})}{\prim{K}} = \prod_{k=1}^{K} \left( 1 + \frac{1}{q_k} \right) < \exp \left(\sum_{k=1}^K \frac{1}{q_k} \right)$, for all $K \in \NN$. Using Lemma~\ref{lem:rosser-schoenfeld}\ref{rs:1}, we can further bound the above expression as follows
\begin{equation*}
\frac{\sigma(\prim{K})}{\prim{K}} < \exp \left(\sum_{k=1}^K \frac{1}{q_k} \right) < e^B \ln(q_K) \exp \left( \frac{1}{(\ln(q_K))^2} \right), \quad K \in \NN.
\end{equation*}

(c) This follows using Lemma~\ref{lem:taylor-remainder}, and noting that $\ln \ln \ln (\prim{3}) > 0$:
\begin{equation*}
C(\prim{K}) = e^{\gamma} \sum_{j=0}^{K} \frac{(\ln \ln \vartheta(q_K))^j}{j!} > e^{\gamma} \ln \vartheta(q_K) \left( 1 - \frac{(\ln \ln \vartheta(q_K))^{K+1}}{(K+1)!} \right), \quad K \ge 3,
\end{equation*} 
where we have used $\ln(\prim{K}) = \vartheta(q_K)$.
\end{proof}

\begin{theorem}[Primorial numbers]
\label{thm:primorials}
$\frac{\sigma(\prim{K})}{\prim{K}} < C(\prim{K})$, for all positive integers $K \ge 4$.
\end{theorem}
\begin{proof}
Set $\ell = 13$, and we divide the proof into two cases: (i) $K \ge \ell$, and (ii) $4 \le K < \ell$. This gives $q_\ell = 41$. We first prove case (i). Assuming $K \ge \ell$, from Lemma~\ref{lem:primorial-helper}\ref{ph:2},\ref{ph:3} it suffices  to prove \[e^B \ln(q_K) \exp \left( \frac{1}{(\ln(q_K))^2} \right) \le e^{\gamma} \ln \vartheta(q_K) \left( 1 - \frac{(\ln \ln \vartheta(q_K))^{K+1}}{(K+1)!} \right),\] or equivalently,
\begin{equation}
\label{eq:primorials-proof-1}
e^{B} \exp \left( \frac{1}{(\ln(q_K))^2} \right) \le e^{\gamma} \frac{\ln \vartheta(q_K)}{\ln(q_K)} \left( 1 - \frac{(\ln \ln \vartheta(q_K))^{K+1}}{(K+1)!} \right).
\end{equation}
By Lemma~\ref{lem:primorial-helper}\ref{ph:1} we have 
\begin{equation}
\label{eq:primorials-proof-2}
\begin{split}
\frac{\ln \vartheta(q_K)}{\ln(q_K)} &> 1 - \frac{\alpha}{(\ln(q_\ell))^2} > 0.9, \quad \alpha := \frac{\ln(q_{13})}{\ln(q_{13}) - 1}.
\end{split}
\end{equation}
Next, we bound the term $\frac{(\ln \ln \vartheta(q_K))^{K+1}}{(K+1)!}$. By Lemma~\ref{lem:rosser-schoenfeld}\ref{rs:3} and Lemma~\ref{lem:log-bounds}\ref{lg:2}, we have $\ln \vartheta(q_K) < \ln (q_K) + 1/(2 \ln(q_K)) $, and so we get
\begin{equation}
\label{eq:primorials-proof-3}
\ln \ln \vartheta(q_K) < \ln \ln (q_K) + \frac{1}{2 (\ln(q_K))^2} < \ln \ln (q_K) + \frac{1}{2 (\ln(q_\ell))^2}.
\end{equation}
Now, we may use Lemma~\ref{lem:rosser-schoenfeld}\ref{rs:4} and Lemma~\ref{lem:log-bounds}\ref{lg:2} again, to obtain
\begin{equation*}
\begin{split}
\ln(q_K) &< \ln K + \ln \ln K + \ln \left (1 + \frac{\ln \ln K}{\ln K} \right) < \ln K + \ln \ln K + \frac{\ln \ln K}{\ln K} \\
&= \ln K + \ln \ln K (1 + 1/\ln K) < \ln K + 2 \ln \ln K,
\end{split}
\end{equation*}
which gives $\ln \ln(q_K) < \ln \ln K +  \frac{2\ln \ln K}{\ln K} = \ln \ln K ( 1 + 2/\ln K) \le \ln \ln K (1 + 2 / \ln \ell)$. Using this bound in \eqref{eq:primorials-proof-3} leads to the estimate
\begin{equation*}
\ln \ln \vartheta(q_K)
<
 \ln \ln K \left( 1 + \frac{2}{\ln \ell} \right) + \frac{1}{2(\ln(q_\ell))^2} < 1.78 \ln \ln K + 0.037.
\end{equation*}
Combining this estimate with the elementary lower bound $(K+1)! > \left(\frac{K+1}{e}\right)^{K+1}$, we get
\begin{equation}
\label{eq:primorials-proof-4}
\frac{(\ln \ln \vartheta(q_K))^{K+1}}{(K+1)!} < \left( \frac{e \ln \ln \vartheta(q_K)}{K+1} \right)^{K+1} < \left( \frac{e (1.78 \ln \ln K + 0.037)}{K+1} \right)^{K+1}.
\end{equation}
Now consider the function $f(x):= \frac{e (1.78 \ln \ln x + 0.037)}{x+1}$, for $x \ge 13$. We have $f(x) > 0$, $f'(x) < 0$, for all $x \ge 13$, and $f(13) < 0.34$; thus $f(x) < 1$. Using these facts about $f$, together with \eqref{eq:primorials-proof-4}, implies
\begin{equation*}
\begin{split}
     \left( 1 - \frac{(\ln \ln \vartheta(q_K))^{K+1}}{(K+1)!} \right) &> 1 - \left( \frac{e (1.78 \ln \ln K + 0.037)}{K+1} \right)^{K+1} > 1 - \left( \frac{e (1.78 \ln \ln K + 0.037)}{K+1} \right)^{14}\\
     &\ge 1 - \left( \frac{e (1.78 \ln \ln \ell + 0.037)}{\ell+1} \right)^{14} > 0.999,
\end{split}
\end{equation*}
and then combining with \eqref{eq:primorials-proof-2} gives
\begin{equation*}
e^{\gamma} \frac{\ln \vartheta(q_K)}{\ln(q_K)} \left( 1 - \frac{(\ln \ln \vartheta(q_K))^{K+1}}{(K+1)!} \right) > e^\gamma \cdot 0.9 \cdot 0.999 > 1.6, \quad K \ge \ell.
\end{equation*}
On the other hand, we also have $e^{B} \exp \left( \frac{1}{(\ln(q_K))^2} \right) \le e^{B} \exp \left( \frac{1}{(\ln(q_\ell))^2} \right) < 1.4$, and thus, we have proved \eqref{eq:primorials-proof-1}.

It remains to prove case (ii). We verify the $\omega-$inequality directly for these primorial numbers $q_K^\#$, for $4 \le K < \ell$, in the numerical check in \ref{list:primorial_check}.
\end{proof}

\subsection{Square-free and odd integers}
\label{ssec:odd-square-free}
There are some easy corollaries of Theorem~\ref{thm:primorials}, which we list below, and they allow us to show that the $\omega-$inequality holds for both square-free integers and odd integers.

\begin{corollary}[Square-free]
\label{cor:square-free}
Let $n > 5040$ be square-free. Then $\frac{\sigma(n)}{n}  < C(n)$.
\end{corollary}

\begin{proof}
Let $n > 5040$ have the prime factorization $n=p_1 \dots p_K$. If $\omega(n) = K \le 3$, then the result follows by Lemma~\ref{lem:three-primes}, so we can assume $K > 3$. Now consider the primorial $\prim{K}$. We get by Theorem~\ref{thm:primorials}, $\frac{\sigma(\prim{K})}{\prim{K}} < C(\prim{K})$. Then the result follows immediately by Lemma~\ref{lem:substitution}.
\end{proof}

\begin{corollary}[Odd integers]
\label{cor:odd-integers}
Let $n > 5040$ be odd. Then $\frac{\sigma(n)}{n} < C(n)$.
\end{corollary}

\begin{proof}
Let $\omega(n) = K$. If $K \le 3$, the result is true by Lemma~\ref{lem:three-primes}, so we may assume $K > 3$. Let $n$ have the prime factorization $n = p_1^{a_1} \dots p_K^{a_K}$. Now by \eqref{eq:sigma-n}, we have
\begin{equation}
\label{eq:odd-integers-proof-1}
\frac{\sigma(n)}{n} = \prod_{k=1}^{K} \left( 1 + \frac{1}{p_k} + \dots + \frac{1}{p_k^{a_k}} \right) < \prod_{k=1}^{K} \frac{p_k}{p_k - 1} = \prod_{k=1}^{K} \left( 1 + \frac{1}{p_k - 1} \right).
\end{equation}
Notice that $p_k - 1 \ge q_k$ for every $k = 1,\dots,K$ (since $n$ is odd, we have $p_1 \ge 3$), and so we get from \eqref{eq:odd-integers-proof-1} the upper bound 
\begin{equation*}
\frac{\sigma(n)}{n} < \prod_{k=1}^{K} \left( 1 + \frac{1}{q_k} \right) = \frac{\sigma(\prim{K})}{\prim{K}} < C(\prim{K}),
\end{equation*}
where the last inequality is by Theorem~\ref{thm:primorials}. Finally, notice that $\omega(\prim{K}) = \omega(n) = K$, and by construction $\prim{K} < n$, and so by Lemma~\ref{lem:monotonicity}, we have $C(\prim{K}) < C(n)$, finishing the proof.
\end{proof}
\section{Minimal counterexamples to the \texorpdfstring{$\pi-$ and $\omega-$}{}inequalities}
\label{sec:minimal-counterex}

We say that a positive integer $n > 5040$ is the \textit{minimal counterexample} to the $\pi-$inequality (resp. $\omega-$inequality) if $n$ is the first integer greater than 5040 for which the $\pi-$inequality (resp. $\omega-$inequality) is violated. In \cite{akbary2009superabundant} Akbary and Friggstad proved that the minimal counterexample $n > 5040$ to Robin's inequality must be a superabundant number. The goal of this section is to prove the same for the $\pi-$ and $\omega-$inequalities.

\subsection{\texorpdfstring{Superabundance for the $\pi-$inequality}{}}
\label{ssec:minimal-counterex-pi}
The argument of Akbary and Friggstad~\cite[Theorem~3]{akbary2009superabundant} extends to the
$\pi-$inequality once the finite range $5041\leq n\leq10080$
is verified.

\begin{lemma}[$\pi-$superabundance]
\label{lem:pi-extremality}
Let $n$ be the minimal counterexample to the $\pi-$inequality. Then  $n$ is a superabundant number. 
\end{lemma}

\begin{proof}
A numerical check verifies the $\pi-$inequality in the range $5041 \le n \le 10080$ (see \ref{list:check-pi-n-1b}). Now, let $n>10080$ be the minimal counterexample to the $\pi-$inequality and suppose $n$ is not superabundant, which respectively implies $\sigma(n)/n \ge  e^{\gamma} \sum_{j=0}^{\pi(n)} \frac{(\ln \ln \ln n)^j}{j!}$, and the existence of $n'< n$ such that $\sigma(n')/n'\ge \sigma(n)/n$. Since $\pi$ is a monotonically non-decreasing function over the positive integers, this furthermore implies that 
\begin{equation}
\label{eq:pi-extremality-1}
    \frac{\sigma(n')}{n'} \ge \frac{\sigma(n)}{n} \ge  e^{\gamma} \sum_{j=0}^{\pi(n)} \frac{(\ln \ln \ln n)^j}{j!} \ge e^{\gamma} \sum_{j=0}^{\pi(n')} \frac{(\ln \ln \ln n')^j}{j!},
\end{equation}
and therefore $n'$ is a smaller counterexample. If $n'>5040$ then we have a contradiction to the minimality of $n$, so we may suppose  $n' \le 5040$. Set $n'' = 10080$, which is a superabundant number, and so we have $\sigma(n'')/n'' > \sigma(n')/n'$. Combined with \eqref{eq:pi-extremality-1} and again using the monotonicity of the $\pi$ function, we have $\sigma(n'')/n''  > e^{\gamma} \sum_{j=0}^{\pi(n)} \frac{(\ln \ln \ln n)^j}{j!} \ge e^{\gamma} \sum_{j=0}^{\pi(n'')} \frac{(\ln \ln \ln n'')^j}{j!}$. Thus $n''$ is a smaller counterexample than $n$, and again gives a contradiction.
\end{proof}

\subsection{\texorpdfstring{Properties of the minimal counterexample to $\omega-$inequality}{}}
\label{ssec:minimal-counterex-omega-prop}
We next derive structural properties of a minimal counterexample to the
$\omega-$inequality, culminating in its superabundance in
Section~\ref{ssec:minimal-counterex-omega}. These properties closely parallel the classical
structure of superabundant numbers. We begin with the following
extremality property.

\begin{lemma}[$\omega-$extremality]
\label{lem:omega-extremality}
Let $n$ be the minimal counterexample to the $\omega-$inequality. Then  $\sigma(k)/k < \sigma(n)/n$, for all $e^e<k < n$, with $\omega(k) \le \omega(n)$. 
\end{lemma}
\begin{proof}
A numerical check verifies the $\omega-$inequality in the range $5041 \le n \le 10080$ (see \ref{list:check-omega-n-1b}). Let $n > 10080$ be the minimal counterexample. For contradiction, assume there exists $e^e < k<n$, with $\omega(k) \le \omega(n)$, for which $\sigma(k)/k \ge \sigma(n)/n$. Since $\omega(k) \le \omega(n)$ and $k<n$, $C(n)>C(k)$ by Lemma~\ref{lem:monotonicity}, leading to the chain of inequalities
\[
\frac{\sigma(k)}{k} \ge \frac{\sigma(n)}{n}\ge C(n)>C(k).
\]
Now, there are two cases: (i) $k > 5040$, and (ii) $k \le 5040$. For case (i), we have a contradiction, since $n$ is minimal. For case (ii), we can simply take $m=10080$, and since 10080 is superabundant, we have $\frac{\sigma(m)}{m}>\frac{\sigma(k)}{k}$. We also have $C(n)> C(m)$ since $n>m$, $\omega(m)=4$, and $\omega(n)\ge 7$ by Lemma~\ref{lem:six-primes}. The above chain is then replaced by
\[
\frac{\sigma(m)}{m} > \frac{\sigma(n)}{n}\ge C(n)>C(m),
\]
which is again a contradiction to minimality of $n$.
\end{proof}
The following two lemmas establish that the minimal counterexample to the $\omega-$inequality has the first two properties in common with superabundant numbers, as mentioned before in Section~\ref{ssec:background}.

\begin{lemma}[Ungapped prime support]
\label{lem:min-counterex-prime-supp}
Let $n$ be the minimal counterexample to the $\omega-$ inequality, and let $K = \omega(n)$. Then the prime support of $n$ contains all the first $K$ prime numbers, that is, $\rad(n) = \prim{K}$. 
\end{lemma}

\begin{proof}
Let $n > 5040$ be the minimal counterexample, and $K = \omega(n)$. We already know that the $\omega-$inequality holds when $K \le 6$, by Lemma~\ref{lem:six-primes}, and that it is also true when $n$ is odd, by Corollary~\ref{cor:odd-integers}. So we may assume that $K > 6$, and that $n$ is an even positive integer throughout this proof.

Let $n$ have the prime factorization $n = p_1^{a_1} \dots p_K^{a_K}$, and suppose for contradiction that $\rad(n) \neq \prim{K}$. Then there exists an integer $1 < j \le K$ such that $p_\ell = q_\ell$, for all $1 \le \ell < j$, and $p_j \ne q_j$ (the inequality $j > 1$ uses the fact that $n$ is even). Note that this implies $p_j > q_j$. Now construct the integer $n' := n q_j^{a_j} / p_j^{a_j}$, and since $\omega(n') = K$ and $n' < n$, this implies $C(n') < C(n)$, by Lemma~\ref{lem:monotonicity}.

Next, notice that by using \eqref{eq:sigma-n}, we have
\begin{equation*}
\frac{\sigma(n)/n}{ \sigma(n')/n'} = \frac{1+p_j^{-1}+p_j^{-2}+\dots+p_j^{-a_j}}{1+q_j^{-1}+q_j^{-2}+\dots+q_j^{-a_j}} < 1,
\end{equation*}
and thus, combining with above, we have
\begin{equation}
\label{eq:min-counterex-prime-supp-proof-1}
\frac{\sigma(n')}{n'} > \frac{\sigma(n)}{n} \ge C(n) > C(n'),
\end{equation}
where the second inequality is by the assumption that $n$ is the minimal counterexample. Now there are two cases possible: (i) $n' > 5040$, and (ii) $n' \le 5040$. In the first case, \eqref{eq:min-counterex-prime-supp-proof-1} gives us a smaller counterexample $n'$ to the $\omega-$inequality, which contradicts the minimality assumption of $n$. 

Finally, consider the second case. Let us construct a new integer $n'':=2^r n' > n'$, where $r:=\text{min}\{r \ge 1~|~2^r n'  > 5040\}$. Then $n''>5040$ by construction. The point of multiplying $n'$ by a power of two is to exploit the known property that $n$ (and hence $n'$) is even, which also ensures $\omega(n'')=K$. Now, there can be two sub-cases: (ii.a) $n'' > n$, and (ii.b) $n'' < n$. For case (ii.a), the definition of $n''$ implies $n'' \le 2 \cdot 5040 = 10080$, which in turn implies $5040 < n < 10080$. But it is known from numerical checks that the $\omega-$inequality holds for all $n \le 10^9$ (see \ref{list:check-omega-n-1b}), thus leading to a contradiction. For case (ii.b), we have
\begin{equation*}
\frac{\sigma(n')/n'}{ \sigma(n'')/n''} = \frac{1+2^{-1}+2^{-2}+\dots+2^{-a_1}}{1+2^{-1}+2^{-2}+\dots+2^{-(a_1 + r)}} < 1,
\end{equation*}
and therefore, combining with \eqref{eq:min-counterex-prime-supp-proof-1}, and Lemma~\ref{lem:monotonicity}, we have $\frac{\sigma(n'')}{n''} > \frac{\sigma(n)}{n} \ge C(n) > C(n'')$, which yields another contradiction to the minimality of $n$, and finishes the proof.
\end{proof}

\begin{lemma}[Falling exponents]
\label{lem:min-counterex-exponents-nondecreasing}
Let $n$ be the minimal counterexample to the $\omega-$inequality, with prime factorization $n = p_1^{a_1} \dots p_K^{a_K}$, and let $K = \omega(n)$. Then $a_1 \ge a_2 \ge \dots \ge a_K$.
\end{lemma}

\begin{proof}
By Lemma~\ref{lem:min-counterex-prime-supp}, we may assume $p_j = q_j$, for all $1 \le j \le K$. Also, we can assume $K > 6$ by Lemma~\ref{lem:six-primes}, and $n > 10080$ by the numerical check in \ref{list:check-omega-n-1b}. Now for contradiction, assume that there exists $1 \le \ell < K$ such that $a_\ell < a_{\ell+1}$, which also implies $a_{\ell+1} \ge 2$. Construct the integer $n' := n q_{\ell} / q_{\ell + 1}$, and note that it satisfies $n' < n$ and $\omega(n') = K$. By Lemma~\ref{lem:monotonicity}, this implies $C(n) > C(n')$. Next, we observe that the ratio
\begin{equation*}
\frac{\sigma(n)/n}{\sigma(n')/n'} = \frac{\left( 1+q_\ell^{-1}+q_\ell^{-2}+\dots+q_\ell^{-a_\ell} \right) \left( 1+q_{\ell+1}^{-1}+q_{\ell+1}^{-2}+\dots+q_{\ell+1}^{-a_{\ell+1}} \right)}{\left( 1+q_\ell^{-1}+q_\ell^{-2}+\dots+q_\ell^{-(a_\ell + 1)} \right) \left( 1+q_{\ell+1}^{-1}+q_{\ell+1}^{-2}+\dots+q_{\ell+1}^{-(a_{\ell+1} - 1)} \right)} < 1,
\end{equation*}
by simply expanding the numerator and denominator, and using the fact that $q_\ell < q_{\ell+1}$, and the assumption $a_{\ell} < a_{\ell + 1}$. Combining the above, we have obtained $\frac{\sigma(n')}{n'} > \frac{\sigma(n)}{n} \ge C(n) > C(n')$. Finally notice that $q_{\ell+1} < 2 q_\ell$ by Lemma~\ref{lem:bertrand} (Bertrand's postulate), which implies $n' > n/2 > 5040$, and thus we have arrived at a contradiction to the minimality of $n$.
\end{proof}
Henceforth, for the rest of this section, we write a minimal counterexample as
\[
    n=q_1^{a_1}\cdots q_K^{a_K},
    \qquad a_1\geq\cdots\geq a_K, \quad K \in \NN.
\]
We next derive quantitative bounds on these multiplicities. The next lemma identifies some more aspects of superabundant numbers (see \cite[Chapter~6]{broughan2017equivalents}, \cite{alaoglu1944highly}), that also hold for the minimal counterexample. Note that part (a) of the lemma in fact implies Lemma~\ref{lem:min-counterex-exponents-nondecreasing}.
\begin{lemma}[Multiplicity bounds]
\label{lem:erdos-results}
Let $n$ be the minimal counterexample to the $\omega-$inequality, with prime factorization $n = q_1^{a_1} \dots q_K^{a_K}$. Then,
\begin{enumerate}[label={(\alph*)}]
    \item $\left \lvert a_j - \left \lfloor a_i \ln q_i / \ln q_j \right \rfloor \right\rvert \le 1$, for all $1 \le i < j \le K$.
    \item $q_j^{a_j} < q_1^{a_1 + 2} = 4 \cdot 2^{a_1}$, for all $1 \le j \le K$.
    \item $a_i \le \left \lfloor (a_j + 2) \ln q_j / \ln q_i \right \rfloor$, for all $1 \le i < j \le K$.
    \item $a_i \ge \max \left \{ \left \lfloor \ln q_j / \ln q_i \right \rfloor, \left \lfloor (a_j - 1) \ln q_j / \ln q_i \right \rfloor + 1 \right \}$, for all $1 \le i < j \le K$.
\end{enumerate}
\end{lemma}

\begin{proof}
We can assume $\omega(n)\ge 7$, by Lemma~\ref{lem:six-primes}. We will construct smaller integers $n'$ in the following proofs, for which $\omega(n')\ge \omega(n)-1$, since we will remove at most one prime from $n$, ensuring $\omega(n')\ge 6$ and therefore $n'\ge 30030$.

The proofs of parts (a) and (b) are exactly the same as \cite[Lemma~6.2]{broughan2017equivalents} and \cite[Lemma~6.4]{broughan2017equivalents} respectively. In the proofs of both cases, one constructs $n' < n$, with $\omega(n') \le \omega(n)$ and $\frac{\sigma(n')}{n'} \ge \frac{\sigma(n)}{n}$, yielding a contradiction to the minimality of $n$. In fact, we also provide a different proof of part (b) in Lemma~\ref{lem:upper-bound}, which slightly improves the bound.

To prove part (c), note that by part (a), we have the implication $\left \lfloor a_i \ln q_i / \ln q_j \right \rfloor \le a_j + 1$. This immediately implies $a_i < (a_j + 2) \ln q_j / \ln q_i$, and the result follows.

For part (d), there are two cases: (i) $a_j \ge 2$, and (ii) $a_j = 1$. For case (i), we use (a) to obtain $\left \lfloor a_i \ln q_i / \ln q_j \right \rfloor \ge a_j - 1$, which implies $a_i > (a_j - 1) \ln q_j / \ln q_i$. Since $a_i$ is an integer and $q_i, q_j$ are distinct primes, this yields the bound $a_i \ge \left \lfloor (a_j - 1) \ln q_j / \ln q_i \right \rfloor + 1$. For case (ii), we use the same construction as in the proof of \cite[Lemma~6.7]{broughan2017equivalents}, and we will prove that $a_i \ge \left \lfloor \ln q_j / \ln q_i \right \rfloor$. Assume for contradiction that $a_i \le \left \lfloor \ln q_j / \ln q_i \right \rfloor  - 1$, which implies $q_i^{a_i + 1} < q_j$. Now consider the integer $n' = n q_i^{a_i + 1} / q_j < n$. Then
\begin{equation*}
\frac{\sigma(n)/n}{\sigma(n')/n'} = \frac{1 + 1/q_j}{1 + 1/q_i^{a_i + 1}}  < 1,
\end{equation*}
which contradicts the minimality of $n$ by Lemma~\ref{lem:omega-extremality}. Combining cases (i) and (ii), the lemma is proved.
\end{proof}
It remains to establish the third structural property, namely $a_K=1$.
Together with Lemmas~\ref{lem:min-counterex-prime-supp}--\ref{lem:min-counterex-exponents-nondecreasing},
this shows that a minimal counterexample satisfies all three
superabundant-number properties listed in Section~\ref{ssec:background}.

\begin{theorem}[Last exponent]
\label{thm:min-counterex-last-exponent}
If \(n = q_1^{a_1} q_2^{a_2} \dots q_K^{a_K} >5040\) is the minimal counterexample to the
$\omega-$inequality, where $K=\omega(n)$, then $a_K = 1$. 
\end{theorem}

\begin{proof}
For contradiction, assume $a_K \ge 2$, and by Lemma~\ref{lem:six-primes}, we can also assume $K\ge 7$. Moreover, by Lemma~\ref{lem:min-counterex-exponents-nondecreasing} we have $a_1 \ge \dots \ge a_K \ge 2$, and thus we have $n > \prod_{k=1}^{7} q_k^2 > 5040^2$. Define $n':= n/q_K$, so that $\omega(n')=K$, and thus we also have $n' > \prod_{k=1}^7 q_k > 5040$. Now since $n$ is the minimal counterexample, and $n' < n$, this implies $\sigma(n)/n \ge C(n)$, and $\sigma(n')/n' < C(n')$, which upon combining yields
\begin{equation}
\label{eq:min-counterex-last-exponent-proof-1}
\frac{C(n)}{C(n')} \le \frac{\sigma(n)/n}{\sigma(n')/n'} = 1+
\frac{q_K-1}{q_K(q_K^{a_K}-1)} \; ,
\end{equation}
where the last equality follows using \eqref{eq:sigma-n}. Our next task is to lower bound $\frac{C(n)}{C(n')} - 1$. Define $\lambda := \ln n$, $\lambda' := \ln n'$, $\mu := \ln \lambda$, and $\mu' := \ln \lambda'$, and note that $\ln \mu > \ln \mu' > 0$ (as $n' > 5040$). By Lemma~\ref{lem:erdos-results}(c) we also have, 
\[
\lambda = \sum_{k=1}^{K}a_k \ln q_k \le a_K \ln q_K + \sum_{k=1}^{K-1} (a_K + 2) \ln q_K = (K a_K + 2K -2) \ln q_K = D \ln q_K,
\]
where we have defined $D := K a_K + 2K -2$. We then obtain the following estimate:
\begin{align*}
\label{eq:min-counterex-last-exponent-proof-2}
\frac{C(n)}{C(n')} - 1 &> \ln \left( \frac{\mu}{\mu'} \right) \left( 1 + \frac{\ln \mu'}{K} \right)^{-1} > \ln \left( \frac{\mu}{\mu'} \right) \left( 1 + \frac{\ln \mu}{K} \right)^{-1} && \text{(by Lemma~\ref{lem:truncated-exponential-ratio}}, \; \mu > \mu')\\
&\ge \ln \left( \frac{\mu}{\mu'} \right) \left( 1 + \frac{\ln \ln (D \ln q_K)}{K} \right)^{-1} && (\text{using }\lambda \le D \ln q_K)\\
&> \frac{\ln q_K}{\lambda \mu} \left( 1 + \frac{\ln \ln (D \ln q_K)}{K} \right)^{-1} && \text{(by Lemma~\ref{lem:log-bounds}\ref{lg:3} twice)}\\
&\ge \frac{1}{D \ln (D \ln q_K)} \left( 1 + \frac{\ln \ln (D \ln q_K)}{K} \right)^{-1}. && (\text{using }\lambda \le D \ln q_K)\\
\end{align*}
Now the right hand side of the above estimate is strictly greater than $\frac{q_K-1}{q_K(q_K^{a_K}-1)}$, by Lemma~\ref{lem:auxiliary-ratio-estimate} proved next, and this shows that $\frac{C(n)}{C(n')} - 1 > \frac{q_K-1}{q_K(q_K^{a_K}-1)}$. This contradicts \eqref{eq:min-counterex-last-exponent-proof-1}, and proves the theorem.
\end{proof}

\begin{lemma}[Auxiliary estimate]
\label{lem:auxiliary-ratio-estimate}
Let $k,a \in \mathbb{N}$, with \(k\ge 7\) and \(a\ge 2\), and let $q_k$ be the $k^{\text{th}}$ prime. Let us define $D:=k(a+2)-2$. Then
\[
D \ln(D \ln q_k) \left(1+\frac{\ln \ln(D \ln q_k) }{k}\right)
< q_k^a < 
\frac{q_k(q_k^a-1)}{q_k-1}.
\]
\end{lemma}

\begin{proof}
For $k \in \mathbb{N}$, $q,r \in \mathbb{R}$ with  $q \ge e^2$, and $r \ge 2$, define the function
\[
G_k(r,q):=
\frac{r(r+2)\ln q}{2q^{r-1}}
\left(1+\frac{\ln(r\ln q)}{k}\right) > 0.
\]
This function is clearly decreasing in $k$. We now claim that it is also decreasing in $q, r$, by equivalently showing the same for $\ln G_k(r,q)$. Computing partial derivatives yield
\begin{equation*}
\begin{split}
q\frac{\partial}{\partial q}\ln G_k(r,q)
&=
\frac1{\ln q}
\left(1+\frac{1}{k+\ln(r\ln q)}\right)-(r-1) < \frac{1+k^{-1}}{\ln q}-(r-1) < 0, \\
r\frac{\partial}{\partial r}\ln G_k(r,q)
&=
\frac{2(r+1)}{r+2}-r\ln q
+\frac1{k+\ln(r\ln q)} < 2 - 2 \ln e^2 + \frac{1}{k} < 0,
\end{split}
\end{equation*}
using the bounds on $q, r$, and this proves the claim. Returning back to the proof of the lemma, this claim implies $G_k(a,q_k) \le G_7(2,17) < 0.84$, since $q_k \ge q_7 = 17$, and thus we get
\begin{equation}
\label{eq:auxiliary-ratio-estimate-proof-1}
\frac{a(a+2)\ln q_k}{2}
\left(1+\frac{\ln(a\ln q_k)}{k}\right)
<q_k^{a-1}.
\end{equation}
Using \eqref{eq:auxiliary-ratio-estimate-proof-1} and $D < k(a+2)$, we can further upper bound $D \ln q_k$ as follows:
\begin{equation*}
D \ln q_k < k(a+2) \ln q_k < \frac{2k q_k^{a-1}}{a \left(1+\frac{\ln(a\ln q_k)}{k}\right)} < kq_k^{a-1} < q_k^a,
\end{equation*}
where in the last two inequalities, we have used $a \ge 2$, and $k \le q_k$. This gives $\ln (D \ln q_k) < a \ln q_k$, and combining everything we finally obtain
\begin{equation*}
\begin{split}
D \ln(D \ln q_k) \left(1+\frac{\ln \ln(D \ln q_k) }{k}\right) &< k(a+2) \cdot a \ln q_k \cdot \left(1+\frac{\ln (a \ln q_k) }{k}\right) \\
&< 2k q_k^{a-1} < q_k^a < \frac{q_k(q_k^a-1)}{q_k-1},
\end{split}
\end{equation*}
where in the second line, we have used the bound from \eqref{eq:auxiliary-ratio-estimate-proof-1} and the fact that $2k < q_k$, when $k \ge 5$, and this finishes the proof.
\end{proof}

Setting $j=K$ in Lemma~\ref{lem:erdos-results}(c),(d) and using
$a_K=1$ gives the following sharper bound:
\begin{corollary}
\label{cor:erdos-results-1}
Let $n$ be the minimal counterexample to the $\omega-$inequality, with prime factorization $n = q_1^{a_1} \dots q_K^{a_K}$. Then, for all $1 \le i < K$, we have
\begin{equation*}
 \left \lfloor \ln q_K / \ln q_i \right \rfloor \le a_i \le \left \lfloor 3\ln q_K / \ln q_i \right \rfloor.
\end{equation*}
\end{corollary}

These bounds also give explicit estimates for $K=\omega(n)$ in terms of $n$.
\begin{corollary}[\texorpdfstring{Lambert $W$ bounds}{}]
\label{cor:n-K-control}
Let $n$ be the minimal counterexample to the $\omega-$inequality, with prime factorization $n = q_1^{a_1} \dots q_K^{a_K}$. Then 
\begin{equation*}
    \frac{\ln n} {6W\!\left(\frac{\ln n}{6}\right)} < K < \frac{\ln n} {W\!\left(\frac{\ln n}{e}\right)}.
\end{equation*}
A weaker lower bound is $K \ge \frac{1}{3} \left(1+\left(1+\frac{3 \ln n}{\ln 2}\right)^{1/2}\right) > \sqrt{\frac{\ln n}{3 \ln 2}}$.
\end{corollary}

\begin{proof}
By Corollary~\ref{cor:erdos-results-1}, we have $a_i \le 3\ln q_K / \ln q_i$, for all $1 \le i < K$, from which it follows that $\ln n=\sum_{i=1}^{K}a_i\ln q_i \le (3K-2)\ln q_K$. First we prove the lower bounds. For the weak lower bound, we simply use Bertrand's postulate (Lemma~\ref{lem:bertrand}) to get $q_K<2^K$, for $K \ge 1$, from which it follows that $\ln n \le K(3K-2)\ln 2$. Rearranging yields $K \ge \frac{1}{3} \left(1+\left(1+\frac{3 \ln n}{\ln 2}\right)^{1/2}\right) > \sqrt{\frac{\ln n}{3 \ln 2}}$, since $K > 0$. For the lower bound in terms of the Lambert $W$ function, we note that $K \ge 7$ by Lemma~\ref{lem:six-primes}; thus by Lemma~\ref{lem:rosser-schoenfeld}\ref{rs:4} we have $q_K \le K(\ln K + \ln \ln K) < K^2$, since $\ln x + \ln \ln x < x$, for all $x > 1$. Therefore, $\ln n <2(3K-2)\ln K <6K\ln K$, and thus $K\ln K>\frac{\ln n}{6}$. We may invert this relationship to get $K > e^{W\left( \frac{\ln n}{6} \right)} = \frac{\ln n}{6 W \left( \frac{\ln n}{6} \right)}$, which proves the lower bound. 

For the upper bound, $n \ge \rad(n)=q^\#_{K} > K! > (K/e)^K$. Thus $\ln n > K(\ln K - 1)$, and since $x \mapsto x(\ln x - 1)$ is strictly increasing for $x \ge 1$,  we may invert this inequality to obtain $K < e^{1 + W\left( \frac{\ln n}{e} \right)} = \frac{\ln n}{W\left( \frac{\ln n}{e} \right)}$. To see this, since $K \ge 6$, we may define $t := \ln K - 1 > 0$, and thus $\ln n > K(\ln K - 1)$ is equivalent to $t e^t < \frac{\ln n}{e}$, which implies $t < W\left( \frac{\ln n}{e} \right)$, and the conclusion follows.
\end{proof}

\subsubsection{Algorithm to extend the $\omega-$inequality range}
\label{sssec:algorithm-range-increase}
Corollary~\ref{cor:erdos-results-1} reduces the search for a minimal counterexample at
fixed $K$ to finitely many exponent tuples. The same strategy applies to Robin's inequality and is known. For $k\geq 7$, define
\begin{equation*}
\begin{split}
    \mathcal{S}_k := \{a = (a_1,\dots,&a_{k-1}, a_k) \in \NN^k : \left \lfloor \frac{\ln q_k}{\ln q_i} \right \rfloor \le a_i \le \left \lfloor \frac{3\ln q_k}{\ln q_i} \right \rfloor, \;a_k=1, \; a_i \ge a_{i+1}, \; \forall  \; 1 \le i < k\}, \\
     &\mathcal{N}_k := \{q_1^{a_1} q_2^{a_2} \cdots q_{k-1}^{a_{k-1}} q_k : (a_1,\dots,a_{k-1},1) \in \mathcal{S}_k\}.
\end{split}
\end{equation*}
By Lemmas~\ref{lem:min-counterex-prime-supp}, \ref{lem:min-counterex-exponents-nondecreasing}, Theorem~\ref{thm:min-counterex-last-exponent}, and Corollary~\ref{cor:erdos-results-1}, 
it then suffices to check the $\omega-$inequality for all integers in the set $\mathtt{N}_K := \bigcup_{k=7}^{K} \mathcal{N}_k$, and if no counterexamples are found, we conclude that $\omega(n) > K$. This implies $n \ge q_{K+1}^\#$, and in fact $n > q_{K+1}^\#$ (by Theorem~\ref{thm:primorials}) and can be slightly improved using Corollary~\ref{cor:erdos-results-1}.
\begin{lemma}
\label{lem:n0-lower-bound}
Suppose that there are no counterexamples to the $\omega-$inequality in $\mathtt{N}_K$, for some $K \ge 7$. Then the $\omega-$inequality holds for all $5040 < n \le n_0$, where $n_0 = q_1^{a_1} q_2^{a_2} \cdots q_K^{a_K} q_{K+1} - 1$, with $a_i = \left \lfloor \frac{\ln q_{K+1}}{\ln q_i} \right \rfloor$, for all $1 \le i < K+1$.
\end{lemma}
\begin{proof}
If $n > 5040$ is the minimal counterexample to the $\omega-$inequality, then $\omega(n) \ge K+1$, since there are no counterexamples in $\mathtt{N}_k$. By Corollary~\ref{cor:erdos-results-1}, this implies $n \ge q_1^{a'_1} q_2^{a'_2} \cdots q_{K+1}^{a'_{K+1}} \cdots q_{\omega(n)}^{a'_{\omega(n)}}$, where $a'_i = \left \lfloor \frac{\ln q_{\omega(n)}}{\ln q_i} \right \rfloor$. Since $\left \lfloor \frac{\ln q_{\omega(n)}}{\ln q_i} \right \rfloor \ge \left \lfloor \frac{\ln q_{K+1}}{\ln q_i} \right \rfloor = a_i$, for all $1 \le i \le K+1$, we conclude $n > n_0$.
\end{proof}

The computation in~\ref{list:check-to-K20} verifies that $\mathtt N_{20}$ contains no
counterexample, and hence gives:

\begin{lemma}
\label{lem:NK-greater-20}
The minimal counterexample $n > 5040$ to the $\omega-$inequality must satisfy $\omega(n) \ge 21$. By Lemma~\ref{lem:n0-lower-bound}, it also satisfies 
\begin{equation*}
\begin{split}
n &>
2^6 \cdot 3^3 \cdot 5^2 \cdot 7^2
\cdot 11 \cdot 13 \cdot 17 \cdot 19 \cdot 23
\cdot 29 \cdot 31 \cdot 37 \cdot 41 \cdot 43
\cdot 47 \cdot 53 \cdot 59 \cdot 61 \cdot 67
\cdot 71 \cdot 73 - 1 \\
&=410555180440430163438262940577599.
\end{split}
\end{equation*}
\end{lemma}

\subsection{\texorpdfstring{Superabundance for the $\omega-$inequality}{}}
\label{ssec:minimal-counterex-omega}

We now prove that a minimal counterexample to the $\omega-$inequality
is superabundant. The first ingredient is an estimate for the
truncated-exponential tail, which will also be used in Section~\ref{sec:equivalance}.

\begin{lemma}[Strong-tail estimate]
\label{lem:strong-tail-estimate}
Let $n > 5040$ be the minimal counterexample to the $\omega-$inequality, with prime factorization $n = q_1^{a_1} \dots q_K^{a_K}$, where $K=\omega(n)$. Let $\tau_K := \ln \ln (6K \ln K)$. Then, we have
\begin{equation*}
    \ln \ln n - \sum_{j=0}^{K}\frac{(\ln \ln \ln n)^j}{j!} < \frac{\tau_K^{K+1}}{(K+1)!} \left( 1 - \frac{\tau_K}{K+2} \right)^{-1} < \frac{6.4 \times 10^{-15}}{\sqrt{6 K \ln K}} < \frac{6.4 \times 10^{-15}}{\sqrt{\ln n}}.
\end{equation*}
\end{lemma}

\begin{proof}
Since $n$ is the minimal counterexample, we can assume $K\ge 21$, by Lemma~\ref{lem:NK-greater-20}. We denote $t:= \ln \ln \ln n > 0$. The proof of Corollary~\ref{cor:n-K-control} already shows that $\ln n < 6K \ln K$ (so the last inequality of this lemma is obvious), and thus $t < \tau_K$. Moreover, $\tau_K < K+2$, for all $K \ge 21$, and thus $t < \tau_K < K+2$. Therefore, we have
\begin{equation}
\label{eq:EK-tail-bound}
\begin{split}
    \ln \ln n - \sum_{j=0}^{K}\frac{(\ln \ln \ln n)^j}{j!} &= e^t - \sum_{j=0}^{K}\frac{t^j}{j!} = \sum_{j=K+1}^{\infty}\frac{t^j}{j!} < \frac{t^{K+1}}{(K+1)!} \left( 1 - \frac{t}{K+2} \right)^{-1} \\
    & < \frac{\tau_K^{K+1}}{(K+1)!} \left( 1 - \frac{\tau_K}{K+2} \right)^{-1},
\end{split}
\end{equation}
which proves the first inequality of this lemma. For the second inequality of the lemma, we define
\begin{equation*}
    R_K :=  \sqrt{6K \ln K} \frac{\tau_K^{K+1}}{(K+1)!} \left( 1 - \frac{\tau_K}{K+2} \right)^{-1},
\end{equation*}
and then it suffices to compute an upper bound on $R_K$. We claim that $R_K \le R_{21}$, for all $K \ge 21$. 

Let us prove this claim. On the domain $x \ge 21$, we introduce the functions $f(x) := \frac{(x+1) \ln (x+1)}{x \ln x}$  and $g(x):= \frac{\ln \ln (6x \ln x)}{x+2}$. Both $f$ and $g$ are positive and strictly decreasing on this domain, and this is easily seen by computing their derivatives. For example, $g'(x) = \frac{g(x)}{x+2} \left( \frac{\left(1 + \frac{1}{\ln x}\right) \left( 1 + \frac{2}{x} \right)}{\ln \ln (6x \ln x) \; \ln (6x \ln x)} - 1\right)$, and since the term in parenthesis is decreasing in $x$, we can upper bound it by evaluating it at  $x=21$. This gives $g'(x) < -0.8 g(x)/(x+2) < 0$, for all $x \ge 21$. Now, we compute the ratio
\begin{equation}
\label{eq:strong-tail-estimate-proof-1}
\begin{split}
    \frac{R_{K+1}}{R_K} &= \sqrt{\frac{(K+1) \ln (K+1)}{K \ln K}} \cdot \left(\frac{\tau_K}{K+2} \right) \cdot \left(\frac{\tau_{K+1}}{\tau_K}\right)^{K+2} \cdot \left( \frac{1-\frac{\tau_K}{K+2}} {1-\frac{\tau_{K+1}}{K+3}} \right) \\
    &\le \sqrt{f(21)} \cdot g(21) \cdot \left(\frac{\tau_{K+1}}{\tau_K}\right)^{K+2} \cdot \frac{1}{1 - g(22)} < 0.087 \left(\frac{\tau_{K+1}}{\tau_K}\right)^{K+2}.
\end{split}
\end{equation}
Next, on $x \ge 21$, let us define the function $\tau(x) := \ln \ln (6x \ln x) > 0$. On this domain, we have
\begin{equation}
\label{eq:strong-tail-estimate-proof-2}
    \frac{\tau'(x)}{\tau(x)} = \frac{1 + \ln x}{x \ln x \; \ln (6x \ln x) \; \ln \ln (6x \ln x)} < \frac{1}{4x},
\end{equation}
where the inequality follows by noticing that $4 \left( 1 + \frac{1}{\ln x} \right) \le 4 \left( 1 + \frac{1}{\ln 21} \right) < 5.32$, for all $x \ge 21$, and similarly $\ln (6x \ln x) \; \ln \ln (6x \ln x) \ge \ln (6 \cdot 21 \cdot \ln 21) \; \ln \ln (6 \cdot 21 \cdot \ln 21) > 10.61$; therefore $4 \left( 1 + \frac{1}{\ln x} \right) < \ln (6x \ln x) \; \ln \ln (6x \ln x)$, for all $x \ge 21$. Integrating \eqref{eq:strong-tail-estimate-proof-2}, we get
\begin{equation*}
    \ln \left( \frac{\tau_{K+1}}{\tau_K} \right) = \int_{K}^{K+1} \frac{\tau'(x)}{\tau(x)} \; dx < \int_{K}^{K+1} \frac{1}{4x} \; dx = \frac{1}{4} \ln \left(1 + \frac{1}{K} \right),
\end{equation*}
from which we conclude that
\begin{equation*}
    \left( \frac{\tau_{K+1}}{\tau_K} \right)^{K+2} < \left( 1 + 1/K\right)^{\frac{K+2}{4}} < e^{\frac{K+2}{4K}} \le e^{\frac{23}{84}} < 1.315, \quad \quad K \ge 21.
\end{equation*}
Plugging this back into \eqref{eq:strong-tail-estimate-proof-1}, we get $\frac{R_{K+1}}{R_K} < 0.087 \cdot 1.315 < 0.115$, and thus the sequence $\{R_K\}_{K \ge 21}$ is strictly decreasing (at a geometric rate), proving the claim. Since $R_{21} < 6.4 \times 10^{-15}$, the lemma is proved.
\end{proof}
 
For later use, set
\begin{equation}
\label{eq:tauK-EK-def}
\tau_K := \ln \ln (6K \ln K), \quad E_K := \frac{\tau_K^{K+1}}{(K+1)!} \left( 1 - \frac{\tau_K}{K+2} \right)^{-1}, \quad K \ge 21.
\end{equation}

The next lemma establishes a lower bound on the $K^{\text{th}}$ primorial number $q_K^\#$, assuming that the minimal counterexample $n$ satisfies $\omega(n)=K$.
\begin{lemma}[Primorial lower bound]
\label{lem:primorial-tail-gap}
Let $n > 5040$ be the minimal counterexample to the $\omega-$inequality, with prime factorization $n = q_1^{a_1} \dots q_K^{a_K}$, where $K=\omega(n)$. Then, we have
\begin{equation*}
    E_K \ln n \; \ln\ln n < \frac{q_K^\#}{n}.
\end{equation*}
\end{lemma}
\begin{proof}
We can assume $K \ge 21$, by Lemma~\ref{lem:NK-greater-20}. Define $t := \ln \ln \ln n > 0$. As $n$ is the minimal counterexample, we have by \eqref{eq:sigma-n}
\begin{equation*}
P_K(t)=\sum_{j=0}^{K}\frac{t^j}{j!}\leq \frac{\sigma(n)}{e^{\gamma}n} = \frac{1}{e^\gamma} \prod_{j=1}^{K} \left( 1 + \frac{1}{q_j} + \dots + \frac{1}{q_j^{a_j}} \right) <
\frac{1}{{e^{\gamma}}}\prod_{j=1}^{K}\frac{q_j}{q_j-1}.
\end{equation*}
Using this, together with Lemma~\ref{lem:strong-tail-estimate}, gives us
\begin{equation*}
\ln \ln n = e^t = P_K(t) + T_K(t) < \frac{1}{{e^{\gamma}}}\prod_{j=1}^{K}\frac{q_j}{q_j-1} + E_K =: Y_K.
\end{equation*}
To prove the lemma, it thus suffices to prove the estimate $E_K Y_K e^{Y_K} e^{e^{Y_K}} < q_K^\#$, or equivalently,
\begin{equation}
\label{eq:primorial-tail-gap-proof-1}
e^{Y_K} + Y_K + \ln Y_K + \ln E_K < \ln q_K^\# = \vartheta(q_K), \quad K \ge 21.
\end{equation}
For the range $21 \le K < 103$, we perform a numerical check ~\ref{list:num_check_K21_50} and we see that \eqref{eq:primorial-tail-gap-proof-1} holds. So let us assume $K \ge 103$, and therefore $q_K \ge 563$, in the rest of this proof.

We first derive two upper bounds on $E_K$ and $\ln E_K$. For $K \ge 103$, we have $\ln (6K \ln K) < \sqrt{K}$, since $f(x):= \sqrt{x} - \ln (6x \ln x)$ satisfies $f(x) > 2.1$, and $f'(x) = \frac{1}{x} \left( \frac{\sqrt{x}}{2} - 1 - \frac{1}{\ln x} \right) > 0$, for $x \ge 103$. Therefore, by \eqref{eq:tauK-EK-def}, we get $\tau_K = \ln \ln (6K \ln K) < \frac{\ln K}{2}$, and thus $\frac{\tau_K}{K+2} < \frac{\ln K}{2 (K+2)} < \frac{1}{2}$. Using \eqref{eq:tauK-EK-def} again gives $E_K < \frac{2\tau_K^{K+1}}{(K+1)!} < 2 \left( \frac{e \tau_K}{K+1} \right)^{K+1} < 2 \left( \frac{e \ln K}{2(K+1)} \right)^{K+1}$, and thus taking logarithms, we get
\begin{equation}
\label{eq:primorial-tail-gap-proof-2}
\ln E_K < \ln 2 + (K+1) \ln \left( \frac{e \ln K}{2(K+1)} \right) < 1 + (K+1) \ln \left( \frac{1}{\sqrt{K}} \right)=1 - \frac{(K+1)}{2} \ln K,
\end{equation}
where we used $\frac{e \ln K}{2(K+1)} < \frac{1}{\sqrt{K}}$, for all $K \ge 103$. Next, we use the bound $q_K \le K(\ln K + \ln \ln K) < K^2$ (see Lemma~\ref{lem:rosser-schoenfeld}\ref{rs:4}), which holds for $K \ge 103$, to get $\ln q_K < 2 \ln K < \sqrt{K}$. Combining with \eqref{eq:primorial-tail-gap-proof-2} this gives $\ln E_K + \ln (2 \ln q_K) < 1 - \frac{(K+1)}{2} \ln K + \ln(2 \sqrt{K}) < 1.7 - \frac{K\ln K}{2} < 0$, and therefore, we get the upper bound
\begin{equation}
\label{eq:primorial-tail-gap-proof-3}
E_K < \frac{1}{2 \ln q_K}, \quad K \ge 103.
\end{equation}

Next, we upper bound $Y_K$ and $e^{Y_K}$. By Lemma~\ref{lem:rosser-schoenfeld}\ref{rs:5}, we have $\frac{1}{{e^{\gamma}}}\prod_{j=1}^{K}\frac{q_j}{q_j-1} < \ln q_K + \frac{1}{2 \ln q_K}$, which after combining with \eqref{eq:primorial-tail-gap-proof-3}, gives $Y_K < \ln q_K + \frac{1}{\ln q_K}$. We conclude that $e^{Y_K} < q_K e^{\frac{1}{\ln q_K}}$, and then the estimate on $\vartheta(q_K)$, from Lemma~\ref{lem:rosser-schoenfeld}\ref{rs:2-1}, gives us
\begin{equation*}
    \vartheta(q_K) - e^{Y_K} > q_K \left(1 - \frac{1}{2\ln q_K} \right) - q_K e^{\frac{1}{\ln q_K}} = -\frac{q_K}{\ln q_K} \left(\frac{1}{2} + \ln q_K \left( e^{\frac{1}{\ln q_K}} -1\right)\right).
\end{equation*}
Since $q_K \ge 563$, we have $\ln q_K > 1$, and thus $e^{\frac{1}{\ln q_K}} -1 < \frac{1}{\ln q_K - 1}$; by Lemma~\ref{lem:rosser-schoenfeld}\ref{rs:6} we have $\frac{q_K}{\ln q_K} < K$; and finally we also have $\ln q_K > \ln(K+1)$; therefore the above inequality yields 
\begin{equation}
\label{eq:primorial-tail-gap-proof-4}
    \vartheta(q_K) - e^{Y_K} > -K \left( \frac{1}{2} +  \frac{ \ln q_K}{\ln q_K - 1}\right) > -K \left( \frac{3}{2} +  \frac{1}{\ln (K+1) - 1}\right).
\end{equation}
The estimates $Y_K < \ln q_K + \frac{1}{\ln q_K}$ and $\ln q_K < \sqrt{K}$, from above, also gives us 
\begin{equation}
\label{eq:primorial-tail-gap-proof-5}
    Y_K + \ln Y_K < 2 \left( \ln q_K + \frac{1}{\ln q_{K}} \right) \le 2 \left( \ln q_K + \frac{1}{\ln q_{103}} \right) < 2 \sqrt{K} + 0.32.
\end{equation}
From \eqref{eq:primorial-tail-gap-proof-2}, \eqref{eq:primorial-tail-gap-proof-4}, \eqref{eq:primorial-tail-gap-proof-5}, in order to prove \eqref{eq:primorial-tail-gap-proof-1}, it thus suffices to show that
\begin{equation*}
     \frac{(K+1)}{2} \ln K - K \left( \frac{3}{2} +  \frac{1}{\ln (K+1) - 1}\right) - 2 \sqrt{K} - 1.32 \ge 0, \quad K \ge 103.
\end{equation*}
But this follows because the function $h(x):= \frac{1}{2}\left(1+\frac{1}{x}\right)\ln x -\frac{3}{2} -\frac{1}{\ln(x+1)-1} -\frac{2}{\sqrt{x}}-\frac{1.32}{x}$, satisfies $h(103) > 0.35$, and $h'(x) > 0$, for all $x \ge 103$. The lemma is proved.
\end{proof}

\begin{theorem}[$\omega-$superabundance]
\label{thm:omega-superabundant}
Let $n > 5040$ be the minimal counterexample to the $\omega-$inequality. Then $n$ is a superabundant number.
\end{theorem}

\begin{proof}
Let $n$ have the prime factorization $n = q_1^{a_1} q_2^{a_2} \dots q_K^{a_K}$, with $\omega(n)=K \ge 21$ (by Lemma~\ref{lem:NK-greater-20}), and for contradiction assume that $n$ is not a superabundant number. Let $n'$ be the superabundant number immediately preceding $n$, and we may assume $5040 < n' < n$ throughout this proof, since the numerical check in \ref{list:check-omega-n-1b} verifies that $n > 10^9$. Since there are no other superabundant numbers in the interval between $n'$ and $n$ (non-inclusive), we must have $\frac{\sigma(n')}{n'} \ge \frac{\sigma(n)}{n}$. Denoting $K':= \omega(n')$, this further implies that $K' > K$, because otherwise we would violate Lemma~\ref{lem:omega-extremality}. 

Now since $n'$ is superabundant, we have $\rad(n') = q_{K'}^\#$ as already mentioned in Section~\ref{sec:intro}; similarly, by Lemma~\ref{lem:min-counterex-prime-supp}, we also have $\rad(n) = q_K^\#$. This then implies that $q_K^\#$ divides $n - n'$, and in particular $n - n' \ge q_K^\#$. Let us also introduce some notation. We define $t:= \ln \ln \ln n$, $t' := \ln \ln \ln n'$, and therefore we have $t > t' > 0$. Since $n$ is the minimal counterexample to the $\omega-$inequality, we have so far $C(n) < C(n')$, or equivalently, $P_K(t) < P_{K'}(t')$, since
\begin{equation*}
    C(n) \le \frac{\sigma(n)}{n} \le \frac{\sigma(n')}{n'} < C(n').
\end{equation*}
It follows that $P_K(t) - P_K(t') < P_{K'}(t') - P_K(t') < T_K(t')$, where $T_K$ was defined in Section~\ref{ssec:background}. Moreover, as $t > t'$, and $K \ge 7$, we also have
\begin{equation*}
\begin{split}
    &P_K(t) - P_K(t') = (\ln \ln \ln n - \ln \ln \ln n') + \sum_{j=2}^{K} \frac{(\ln \ln \ln n)^j - (\ln \ln \ln n')^j}{j!} \\
    &> (\ln \ln \ln n - \ln \ln \ln n') = \int_{n'}^{n} \frac{1}{x \ln x \; \ln \ln x } \; dx > \frac{n - n'}{n \ln n \; \ln \ln n} \ge \frac{q_K^\#}{n \ln n \; \ln \ln n},
\end{split}
\end{equation*}
where in the last inequality on the second line, we used from above that $n - n' \ge q_K^\#$. Therefore, combining everything, we get that $\frac{q_K^\#}{n \ln n \; \ln \ln n} < T_K(t') < T_K(t) < E_K$, where the second-last inequality is due to the monotonicity of $T_K$, and the last inequality is by \eqref{eq:EK-tail-bound}. This contradicts Lemma~\ref{lem:primorial-tail-gap}, and thus we have proved that $n$ is superabundant.
\end{proof}
\section{Equivalence with the Riemann hypothesis}
\label{sec:equivalance}

The $\omega-$inequality trivially implies Robin's inequality, which in turn is equivalent to
Riemann hypothesis~\cite{robin1984grandes}, so it remains only to prove the converse. We give
a proof that the Riemann hypothesis implies the $\omega-$inequality using 
 the minimal counterexample results of Section~\ref{ssec:minimal-counterex-omega}. The proof uses the following
conditional bound on the divisor ratio due to Nicolas~\cite[Corollary~1.2]{nicolas2022sum}:
\begin{lemma}[Nicolas' bound under RH]
\label{lem:Nicolas-bound}
If the Riemann hypothesis is true, then for all $n > 5040$,
\begin{equation}
\label{eq:Nicolas-bound}
    \frac{\sigma(n)}{n} < e^\gamma \left( \ln \ln n - \frac{0.095}{\sqrt{\ln n}} \right).
\end{equation}
\end{lemma}

The proof combines Nicolas' bound with the strong-tail estimate
of Lemma~\ref{lem:strong-tail-estimate}.

\begin{theorem}[Equivalence with RH]
\label{thm:rh-equiv-1}
The $\omega-$inequality is equivalent to the Riemann hypothesis.
\end{theorem}
\begin{proof}
Let us assume that RH is true. For contradiction, assume that the $\omega-$inequality is false, let $n > 5040$ be the minimal counterexample to the $\omega-$inequality, and let $\omega(n)=K$. Then Nicolas' bound \eqref{eq:Nicolas-bound} holds for this $n$, and we get
\begin{equation*}
    \frac{\sigma(n)}{n} < e^\gamma \left( \ln \ln n - \frac{0.095}{\sqrt{\ln n}} \right) < e^\gamma \left( \ln \ln n - \frac{6.4 \times 10^{-15}}{\sqrt{\ln n}} \right) < e^\gamma \sum_{j=0}^{K} \frac{(\ln \ln \ln n)^j}{j!} =  C(n),
\end{equation*}
where the last inequality is by Lemma~\ref{lem:strong-tail-estimate}. This is a contradiction, which proves the theorem, as the other direction is obvious.
\end{proof}

An immediate consequence of Theorem~\ref{thm:rh-equiv-1} is the following:
\begin{corollary}
\label{cor:omega-ca-numbers}
The Riemann hypothesis is true if and only if the $\omega-$inequality holds for all colossally abundant numbers strictly greater than $5040$.
\end{corollary}
\begin{proof}
By Theorem~\ref{thm:rh-equiv-1}, if RH is true, then the $\omega-$inequality holds for all $n > 5040$, and hence also when $n$ is a colossally abundant number. For the other direction, if the $\omega-$inequality holds for all colossally abundant numbers, then Robin's inequality is also true for all colossally abundant numbers \cite[Lemma~7.1]{broughan2017equivalents}. But this implies that Robin's inequality is true for all $n > 5040$, and hence the Riemann hypothesis holds.
\end{proof}
\section{Discussion}

The $\omega$-inequality gives a finite-truncation criterion equivalent to Robin's inequality and hence to the Riemann hypothesis. More generally, any integer-valued truncation $A(n)$ in~\eqref{eq:existence-An} satisfying
\[
A(n)\geq\omega(n)
\]
lies between the $\omega$-inequality and Robin's inequality and is therefore again equivalent to RH; the $\pi$-inequality is one example.

The $\omega$-inequality arose from an extension of the study of inequalities in~\cite{mishra2023mathematical} and was first supported by limited numerical evidence.
The results developed here were obtained in stages. We first established the unconditional cases (primorials in Theorem~\ref{thm:primorials}, square-free integers in Corollary~\ref{cor:square-free}, and odd integers in Corollary~\ref{cor:odd-integers}) and the multiplicity bounds for a hypothetical minimal counterexample (Lemma~\ref{lem:erdos-results}), except for the final condition $a_K=1$ derived in Theorem~\ref{thm:min-counterex-last-exponent}. Proving $a_K=1$ led to a proof of equivalence with Robin's inequality~(Theorem~\ref{thm:rh-equiv-1}). The result $a_K=1$ remains important for a different reason: for each fixed $K$, it reduces the relevant search to a finite set of multiplicity tuples. This finite reduction, together with bounds for the truncated exponential that also entered the equivalence proof, is used in proving that a minimal counterexample must be superabundant.

The scale $\omega(n)$ is also close to a natural lower threshold. Indeed, replacing $\omega(n)$ by~$\omega(n)-1$ already gives infinitely many counterexamples. For $n=2^a$ one has $A(n)=\omega(n)-1=0$, so the corresponding right-hand side is $e^\gamma$, whereas
\[
\frac{\sigma(2^a)}{2^a}
=2-\frac1{2^a}>e^\gamma
\]
for every $a\geq3$. Thus every $2^a>5040$ is a counterexample. The same argument applies to
$A(n)=\lfloor c\,\omega(n)\rfloor$ for every fixed $0<c<1$, since
$A(2^a)=0$. Hence no uniform proportional reduction of the truncation order is possible, although this does not establish pointwise minimality of $\omega(n)$ among all arithmetic truncation rules. 

A further consequence of the present analysis is that, if the $\omega$-inequality fails, its minimal counterexample is superabundant. We record this in Theorem~\ref{thm:omega-superabundant}. Hence it is sufficient to verify the inequality on superabundant integers, or using Corollary~\ref{cor:omega-ca-numbers}, on the even smaller class of colossally abundant integers. This places the truncated criterion within the same extremal framework that governs maximal values of $\sigma(n)/n$, while retaining the additional parameter $K=\omega(n)$. For fixed $K$, the problem is reduced to a finite tuple of non-increasing prime multiplicities subject to the structural bounds obtained above.

These results suggest two related directions. First, one may ask how far the truncation order can be reduced while preserving equivalence with Robin's inequality, and whether there is a natural minimal arithmetic choice of $A(n)$. Second, the finite-dimensional problem at fixed $K$ may admit a sharper extremal description, for example through explicit optimisation of $\sigma(n)/n$ under size and multiplicity constraints, or through a sparse extremal sequence adapted specifically to the $\omega$-truncation. Such a reduction would go beyond the equivalence itself and may clarify the ways in which the finite-truncation formulation provides additional leverage on Robin's criterion.
\subsection*{Acknowledgments}
\label{ssec:ack}

The authors thank Subhayan Roy Moulik, Tom Oliver, Minhyong Kim and participants at the \textit{AI $\times$ Mathematics 2026} workshop at ICMS, Edinburgh, for helpful discussions. C.M. acknowledges support by the \textit{Accelerate Program for Scientific Discovery} and Turing AI grant G111021 at the Department of Computer Science and Technology, University of Cambridge. R.S. acknowledges support from the U.S. Department of Energy, Office of Science, Accelerated Research in Quantum Computing Centers, Quantum Utility through Advanced Computational Quantum Algorithms, grant no. DE-SC0025572.

\subsection*{On LLM use}
\label{ssec:ai-use}
Large language models (LLMs) were consulted asynchronously to identify typos, check calculations, discuss proof sketches, identify errors in algebra, or other redundancies in the manuscript, and any suggestions were reviewed and implemented serially and by hand if correct. 

\appendix
\section{Supporting results and other related bounds}
\label{app:supporting-results}

\subsection{Bounds on exponentials, logarithms, and truncated exponential sums}
\label{ssec:exp-log-bounds}
We collect some analytic estimates used in the main body of the paper.

\begin{lemma}[Taylor remainder]
\label{lem:taylor-remainder}
Let $f: \RR \rightarrow \RR$ be a $(k+1)$-times differentiable function on the closed interval $[a,b]$, and let $x, y \in (a,b)$ be distinct. Then there exists $\xi$ between $x,y$ such that $f(x) = \sum_{j=0}^{k} f^{(j)}(y) \frac{(x-y)^j}{j!} + f^{(k+1)}(\xi) \frac{(x-y)^{k+1}}{(k+1)!}$.  In particular, for the function $f: x \mapsto e^x$, we have for all $x > 0$ and $k \in \NN$, the estimate $e^x \left(1 - \frac{x^{k+1}}{(k+1)!} \right)< \sum_{j=0}^k \frac{x^j}{j!}$.
\end{lemma}

\begin{lemma}[Log-bounds]
\label{lem:log-bounds}
The following statements are true:
\begin{enumerate}[label={(\alph*)}]
    \item \label{lg:1}$\ln(1-x) + \frac{x}{1-y} > 0$, for all $0 < x \le y < 1$.
    \item \label{lg:2} $\ln (1 + x) < x$, for all $0 \neq x > -1$.
    \item \label{lg:3} $\ln \left( \frac{x}{y} \right) > \frac{x-y}{x}$, for all $x > y > 0$.
\end{enumerate}
\end{lemma}

\begin{proof}
(a) Let $f_y(x) := \ln(1-x) + \frac{x}{1-y}$. Then we have $f_y(0) = 0$, and $f'_y(x) = \frac{y-x}{(1-y)(1-x)} > 0$, for $x < y$. (b) Follows from $e^x > 1+x$, when $x \neq 0$. (c) We apply (b) to $\ln \left( \frac{y}{x}\right)$, noting that $\frac{y-x}{x} > -1$, which gives $\ln \left( \frac{y}{x}\right) = \ln \left( 1 + \frac{y-x}{x}\right) < \frac{y-x}{x}$.
\end{proof}

\begin{lemma}[Truncated-exponential ratio bound]
\label{lem:truncated-exponential-ratio}
Let $x > y > 0$. Then for every $N \in \mathbb{N}$,
\begin{equation*}
\frac{P_N(x)}{P_N(y)} > 1 + \frac{x - y}{1 + y/N}.
\end{equation*}
\end{lemma}

\begin{proof}
Fix $N \in \mathbb{N}$. We have $P_N(x) - P_N(y) = \int_{y}^{x} P_{N-1}(u) \;du > (x - y) P_{N-1}(y)$ by the fundamental theorem of calculus, where in the last inequality we have used the monotonicity of $P_{N-1}$. This immediately gives the bound
\begin{equation*}
\begin{split}
    \frac{P_N(x)}{P_N(y)} - 1 &=  \frac{P_N(x) - P_N(y)}{P_N(y)} > (x - y) \frac{P_{N-1}(y)}{P_N(y)} = (x - y) \left( 1 + \frac{y^{N}}{P_{N-1}(y) N!} \right)^{-1} \\
    & \ge (x - y) \left( 1 + \frac{y^{N} (N-1)!}{y^{N-1} N!} \right)^{-1} = (x - y) \left( 1 + \frac{y}{N} \right)^{-1},
\end{split}
\end{equation*}
where we used the fact that $P_{N-1}(y) \ge \frac{y^{N-1}}{(N-1)!}$ in the second line.
\end{proof}

\subsection{Additional multiplicity bounds}
\label{ssec:additional-multiplicity-bounds}
We record some further multiplicity bounds, valid for the minimal counterexample to the $\omega-$inequality and for superabundant numbers.

\vspace{0.3cm}
\begin{lemma}[Multiplicity upper bound]
\label{lem:upper-bound}
Let $n = q_1^{a_1} \dots q_K^{a_K} >5040 $ be either the minimal counterexample to the $\omega-$inequality or a superabundant number with $\omega(n)\ge 7$. Then \[q_{i}^{a_i} \le 2^{a_1+2}\Xi(q_i) - \Delta(q_i) < 2^{a_1 + 2},\] for all $2 \le i \le K$, where $\Xi(q):=\frac{1-\frac1q}{2-2^{1-s_q}}$,
$\Delta(q):=\frac{1-\frac{2^{s_q}}q}{2^{s_q}-1}$, and $s_{q}:=\lfloor \log_2(q-1)\rfloor$. The first inequality is strict in the superabundant case.
\end{lemma}

\begin{proof}
If $n$ is the minimal counterexample to the $\omega-$inequality, then by Lemma~\ref{lem:six-primes}, $\omega(n) \ge 7$. Thus $q_i$ must be an odd prime for any $2 \le i \le K$, which gives $1 \le s_{q_i} =\lfloor \log_2({q_i}-1)\rfloor= \lfloor \log_2(q_i)\rfloor$. This implies $2 \le 2^{s_{q_i}} < q_i$, and  it follows that $\Xi(q_i) < 1$ and $\Delta(q_i) > 0$; so the inequality $2^{a_1+2}\Xi(q_i) - \Delta(q_i) < 2^{a_1 + 2}$ is clear.

Now  set \(n':=n\Big(\frac{2^{s_{q_i}}}{q_i}\Big),\) so we have \(n'<n\), and we note $\omega(n')\ge 6$, which implies $n' \ge 30030$. If $n$ is superabundant then $\frac{\sigma(n')}{n'} < \frac{\sigma(n)}{n}$, while $\frac{\sigma(n')}{n'} \le \frac{\sigma(n)}{n}$ if $n$ is the minimal counterexample to the $\omega-$inequality, by Lemma~\ref{lem:omega-extremality}. 
A calculation shows that the inequality
\[
\frac{\sigma(n')/n'}{\sigma(n)/n}
=
\left(\frac{1-2^{-(a_1+s+1)}}{1-2^{-(a_1+1)}} \right)
\left(\frac{1-q_i^{-a_i}}{1-q_i^{-(a_i+1)}}\right) \le 1
\]
is true if and only if the inequality $q_i^{a_i} \le 2^{a_1+2}\Xi(q_i)-\Delta(q_i)$ is true as well, with the same equivalence also holding when both the inequalities are strict. This proves the lemma.
\end{proof}

\begin{lemma}[Largest abundant index]
\label{lem:J-bounds}
Let $n>5040$ be either the minimal counterexample to the $\omega-$inequality or a superabundant number with $\omega(n)\ge 7$, with prime factorization $n = q_1^{a_1} \dots q_K^{a_K}$. Denote the index of the largest prime with non unit multiplicity as $J:=\max\{j: a_j\ge 2\}$, such that $a_J\ge 2$, and $a_{J+1}=\cdots=a_K=1$. Then,
\[
\pi\!\left(2^{a_1/3}\right)-1 \;<\; J \; \le \; \pi\!\left(2^{a_1/2+1}\right).
\]
\end{lemma}

\begin{proof}
We have already proved that the minimal counterexample to the $\omega-$inequality is a superabundant number in Theorem~\ref{thm:omega-superabundant}, so we just prove the lemma assuming $n$ is superabundant. Note that all superabundant numbers greater than $30$ are divisible by $4$. If this is not the case (i.e. if $a_1=1$), then $n$ must be the primorial number $q_K^\#$, by the properties listed in Section~\ref{ssec:background}. As $n > 36$ and $q_4^\# = 210$ is not superabundant, this also implies $K > 4$, or $q_K \ge 11$, and moreover, $a_K=1$ by \cite[Theorem 3]{alaoglu1944highly}. Then Lemma~\ref{lem:upper-bound} gives $q_K < 8$, which is a contradiction. Thus the set $\{j: a_j \ge 2\}$ is non-empty.

Next, as $a_K=1$, and the multiplicity exponents are non-increasing (i.e. $a_1 \ge a_2 \ge \cdots \ge a_K$), using \cite[Theorem 1]{alaoglu1944highly}, we have $1 \le J \le K-1$. From this, it follows by Lemma~\ref{lem:upper-bound}, that $q_J^2\le q_J^{a_J}<4\cdot 2^{a_1}$, and therefore $q_J<2^{a_1/2+1}$. Since $\pi$ is a non-decreasing function, this implies the upper bound $J \le \pi\!\left(2^{a_1/2+1}\right)$. For the lower bound, we use Lemma~\ref{lem:erdos-results}(a) and  $a_{J+1}=1$, which gives $\left\lfloor \frac{a_1 \ln 2}{\ln q_{J+1}}\right\rfloor\le 2$. Hence, $a_1\frac{\ln 2}{\ln q_{J+1}}<3$, or equivalently $q_{J+1}>2^{a_1/3}$. Again by monotonicity of $\pi$, we conclude $J+1 > \pi(2^{a_1/3})$, completing the proof.
\end{proof}

\begin{lemma}[Multiplicity bound using $q_{K+1}$]
\label{lemma:mult_boundSA2}
Let $n = q_1^{a_1} \dots q_K^{a_K}$ be a superabundant number, let $1 \le j \le K$, and let $\alpha \in \NN$. If  $q_j^\alpha>q_{K+1}+1$, then $a_j\le 2(\alpha-1)$.
\end{lemma}
\begin{proof}
Suppose $q_j^\alpha>q_{K+1}+1$, so we may assume $\alpha \ge 2$ (as $q_j \le q_K < q_{K+1}$), and for contradiction assume $a_j > 2(\alpha - 1) > \alpha$. We construct a smaller integer $n'$ with a greater divisor ratio; let $n' := n\frac{q_{K+1}}{q_j^{\alpha}} < n$. We compute the ratio $\frac{\sigma(n')/n'}{\sigma(n)/n}$:
\begin{align*}
\frac{\sigma(n')/n'}{\sigma(n)/n}
&=
\frac{\sigma(q_j^{a_j-\alpha})/q_j^{a_j-\alpha}}{\sigma(q_j^{a_j})/q_j^{a_j}}
\cdot
\frac{\sigma(q_{K+1})}{q_{K+1}} = 
\frac{q_j^{a_j-\alpha+1}-1}{q_j^{a_j+1}-1}\cdot q_j^{\alpha}\cdot\Bigl(1+\frac{1}{q_{K+1}}\Bigr).
\end{align*}
Since $n$ is superabundant, we must have $\frac{\sigma(n')/n'}{\sigma(n)/n} < 1$, which implies $q_j^{a_j+1}-q_j^{\alpha} < q_{K+1}\bigl(q_j^{\alpha}-1\bigr)$. 
Define $t:=\lfloor\log _{q_j}(1+q_{K+1})\rfloor$; thus $t \le \log _{q_j}(1+q_{K+1})< t+1$, or equivalently, $q_j^{t}\le 1+q_{K+1}<q_j^{t+1}$.  We claim that $a_j < \alpha + t$. To see this, assume otherwise that $a_j \ge \alpha + t$. Then $q_j^{a_j-\alpha+1}-1\ge q_j^{t+1}-1> q_{K+1}$, or equivalently, $q_j^{a_j+1}-q_j^\alpha > q_{K+1} q_j^\alpha >  q_{K+1}(q_j^\alpha-1)$, which cannot be true, proving the claim. Moreover, we further have $\lfloor\log _{q_j}(1+q_{K+1})\rfloor\le \alpha-1$, since $q_j^\alpha>q_{K+1}+1$, and since $\alpha$ is an integer. Therefore, combining we get,
\begin{equation*}
a_j < \alpha + t = \alpha + \lfloor\log _{q_j}(1+q_{K+1})\rfloor \le 2 \alpha - 1,
\end{equation*}
from which it follows that $a_j \le 2(\alpha - 1)$, since both sides are integers.
\end{proof}

For the above lemma, note the special case $\alpha =2$, which forces $a_j\in\{1,2\}$ whenever $q_j^2 > q_{K+1}+1$.

\section{Numerical data and list of finite checks}
\label{app:finite-checks}
We list the finite computations used in the proofs together with the
code required to reproduce them. Table~\ref{tab:finite-checks} summarizes their usage.

\begin{enumerate}[label=(L.\arabic*)]
    \item \label{list:power-three-primes} \textit{Arbitrary powers of at most three primes}. With $r_K$ as defined in proof of Lemma~\ref{lem:three-primes}, we have $r_K = 2$, $3$, $3.75$, and the computed values of $e^\gamma \sum_{k=0}^{K} \frac{0.7622^k}{k!}$ are $3.138$, $3.656$, $3.787$, for $K=1,2,3$ respectively.
    \item \label{list:power-six-primes} \textit{Arbitrary powers of at most six primes}. In proof of Lemma~\ref{lem:six-primes}, we have $r_K = 4.375$, $4.8125$, $5.2135$, and the computed values of $e^\gamma \sum_{k=0}^{K} \frac{1.076^k}{k!}$ are $5.201$, $5.223$, $5.227$, for $K=4,5,6$ respectively.
    \item \label{list:check-omega-n-1b} \textit{Check $\omega-$inequality up to $n \le 10^9$}. The following \textsc{Mathematica} code (v15.0) was used to verify this for every $n\in (5040,10^9]$:
    \begin{verbatim}
DivisorSigma[1, n]/(Exp[EulerGamma] n) <
  Sum[(Log[Log[Log[n]]])^j/j!, {j, 0, PrimeNu[n]}]
\end{verbatim}
    \item \label{list:check-pi-n-1b} \textit{Check $\pi-$inequality for $5041\le n \le 10080$}. The following \textsc{Mathematica} code (v15.0) was used to verify this for every $n\in (5040,10080]$:
\begin{verbatim}
DivisorSigma[1, n]/(Exp[EulerGamma] n) < 
  Sum[(Log[Log[Log[n]]])^j/j!, {j, 0, PrimePi[n]}]
\end{verbatim}
\item \label{list:check-to-K20}\textit{Check $\omega-$inequality on minimal counterexample candidates with $\omega(n) \le 20$}. \textsc{Mathematica} code (v15.0) was used to verify this for every $n$ with $7\le \omega(n) \le 20$. 
\item \label{list:num_check_K21_50}\textit{Check \(e^{Y_K} + Y_K + \ln Y_K + \ln E_K < \ln q_K^\# = \vartheta(q_K)\) for $21\le K<103$ in  Lemma~\ref{lem:primorial-tail-gap}}. The following \textsc{Mathematica} code (v15.0) was used to conduct the check:
\begin{flushleft}
\small\ttfamily
check[$\omega$\_] := Module[\{$\theta$, $\tau$, E, Y, k\},\\
\hspace*{1em}$\theta$[k\_] := Total[Log[Prime[\#]] \& /@ Range[k]];\\
\hspace*{1em}$\tau$[k\_] := Log[Log[6 k Log[k]]];\\
\hspace*{1em}E[k\_] :=
  ($\tau$[k]\textasciicircum(k + 1)/(k + 1)!)/
  (1 - $\tau$[k]/(k + 2));\\
\hspace*{1em}Y[k\_] := E[k] +
  Exp[-EulerGamma] Times @@\\
\hspace*{2em}(Prime[\#]/(Prime[\#] - 1) \& /@ Range[k]);\\[2pt]
\hspace*{1em}N[Exp[Y[$\omega$]] + Y[$\omega$] + Log[E[$\omega$]] +
  Log[Y[$\omega$]] $<$ $\theta$[$\omega$], 80]];\\[2pt]
check /@ Range[21, 102]
\end{flushleft}
\item \label{list:primorial_check} \textit{Check $\omega-$inequality for some primorials}. \textsc{Mathematica} code (v15.0) was used to verify this for all $n=q_K^\#$ with $4\le K< 13$
\begin{verbatim}
Check[n_] := DivisorSigma[1, n]/(Exp[EulerGamma] n) < 
    Sum[(Log[Log[Log[n]]])^j/j!, {j, 0, PrimeNu[n]}];
Primorial[k_] := Times @@ (Prime /@ Range[k]);
ns = Primorial[#] & /@ Range[4, 12];
Check /@ ns
\end{verbatim}
\end{enumerate}

\begin{table}[h]
\centering
\begin{tabular}{|l|c|c|}
\hline
\textbf{Claim number} & \textbf{Program name} & \textbf{Usage} \\
\hline
\ref{list:power-three-primes} Powers of at most 3 primes & threePrimes.py & Lemma~\ref{lem:three-primes} \\
\hline
\ref{list:power-six-primes} Powers of at most 6 primes & sixPrimes.py & Lemma~\ref{lem:six-primes} \\
\hline
\ref{list:check-omega-n-1b} Check $\omega-$inequality ($n \le 10^9$) & see code above & Lemma~\ref{lem:six-primes} \\
\hline
\ref{list:check-pi-n-1b} Check $\pi-$inequality ($5041\le n \le 10080$) & see code above  & Lemma~\ref{lem:pi-extremality} \\
\hline
\ref{list:check-to-K20} Check $\omega-$inequality minimal counterexamples & checkToK20.nb  & Lemma~\ref{lem:NK-greater-20} \\
\hline
\ref{list:num_check_K21_50} Check Equation~\eqref{eq:primorial-tail-gap-proof-1} for $21\le K<103$ & see code above  & Lemma~\ref{lem:primorial-tail-gap} \\
\hline
\ref{list:primorial_check} Check $\omega-$inequality for  primorials & see code above  & Theorem~\ref{thm:primorials} \\
\hline
\end{tabular}
\caption{Numerical claims and corresponding program names for reproducibility, and first usage. The actual codes are available at this \href{https://github.com/rsarkar-github/prime-numb3rs/tree/main/paper/scripts}{GitHub} link.}
\label{tab:finite-checks}
\end{table}

\bibliographystyle{amsalpha}
\bibliography{bibliography.bib}

@article{mishra2023mathematical,
  title={Mathematical conjecture generation using machine intelligence},
  author={Mishra, Challenger and Moulik, Subhayan Roy and Sarkar, Rahul},
  journal={arXiv preprint arXiv:2306.07277},
  year={2023}
}

@article{ramanujan1919proof,
  title={\href{https://ramanujan.sirinudi.org/Volumes/published/ram24.html}{A proof of {B}ertrand's postulate}},
  author={Ramanujan, Srinivasa},
  journal={Journal of the Indian Mathematical Society},
  volume={11},
  pages={181--182},
  year={1919}
}

@article{alaoglu1944highly,
  title={\href{https://doi.org/10.1090/S0002-9947-1944-0011087-2}{On highly composite and similar numbers}},
  author={Alaoglu, Leonidas and Erd{\"o}s, Paul},
  journal={Transactions of the American Mathematical Society},
  volume={56},
  number={3},
  pages={448--469},
  year={1944},
  publisher={JSTOR}
}

@misc{OEIS_A067698,
  author       = {{OEIS Foundation Inc.}},
  title        = {The online encyclopedia of integer sequences},
  howpublished = {\url{https://oeis.org/A067698}},
  year         = {Sequence A067698},
}

@article{akbary2009superabundant,
  title={\href{https://doi.org/10.4169/193009709X470128}{Superabundant numbers and the {R}iemann hypothesis}},
  author={Akbary, Amir and Friggstad, Zachary},
  journal={The American Mathematical Monthly},
  volume={116},
  number={3},
  pages={273--275},
  year={2009},
  publisher={Taylor \& Francis}
}

@article{ramanujan1915highly,
  title={\href{https://londmathsoc.onlinelibrary.wiley.com/doi/10.1112/plms/s2_14.1.347}{Highly composite numbers}},
  author={Ramanujan, Srinivasa},
  journal={Proceedings of the London Mathematical Society},
  volume={2},
  number={1},
  pages={347--409},
  year={1915},
  publisher={Oxford University Press}
}

@article{axler2023robin,
  title={\href{https://doi.org/10.1007/s11139-022-00683-0}{On {R}obin’s inequality}},
  author={Axler, Christian},
  journal={The Ramanujan Journal},
  volume={61},
  number={3},
  pages={909--919},
  year={2023},
  publisher={Springer},
}

@article{nicolas2022sum,
  title={\href{https://doi.org/10.1007/s11139-021-00491-y}{The sum of divisors function and the {R}iemann hypothesis}},
  author={Nicolas, Jean-Louis},
  journal={The Ramanujan Journal},
  volume={58},
  number={4},
  pages={1113--1157},
  year={2022},
  publisher={Springer}
}

@article{banks2009nicolas,
  title={\href{https://doi.org/10.1007/s00605-008-0022-x}{The {N}icolas and {R}obin inequalities with sums of two squares}},
  author={Banks, William D and Hart, Derrick N and Moree, Pieter and Wesley Nevans, C},
  journal={Monatshefte f{\"u}r Mathematik},
  volume={157},
  number={4},
  pages={303--322},
  year={2009},
  publisher={Springer}
}

@article{broughan2015robin,
  title={\href{https://doi.org/10.5281/zenodo.10456122}{{R}obin's inequality for 11-free integers}},
  author={Broughan, Kevin A and Trudgian, Tim},
  year={2015},
  journal={Integers},
  volume={15},
  number={A12}
}

@article{platt2021robin,
  title={\href{https://doi.org/10.5281/zenodo.10807556}{{R}obin's inequality for 20-free integers}},
  author={Platt, David J and Morrill, Thomas},
  journal={Integers: Electronic Journal of Combinatorial Number Theory},
  volume={21},
  pages={28},
  year={2021},
  publisher={University of West Georgia}
}

@article{sole2012robin,
  title={\href{https://doi.org/10.1515/integ.2011.103}{The {R}obin inequality for 7-free integers}},
  author={Sol{\'e}, Patrick and Planat, Michel},
  year={2012},
  journal={Integers},
  volume={12.2},
  pages={301-309},
}

@article{luca2025robin,
  title={\href{https://doi.org/10.5281/zenodo.14907242}{On {R}obin's inequality}},
  author={Luca, Florian and Sol{\'e}, Patrick},
  journal={Integers: Electronic Journal of Combinatorial Number Theory},
  volume={25},
  pages={A8},
  year={2025}
}

@article{rosser1962approximate,
title={\href{https://doi.org/10.1215/ijm/1255631807}{Approximate formulas for some functions of prime numbers}},
  author={Rosser, J Barkley and Schoenfeld, Lowell},
  journal={Illinois Journal of Mathematics},
  volume={6},
  number={1},
  pages={64--94},
  year={1962},
  publisher={Duke University Press}
}

@article{robin1984grandes,
  title={\href{https://zbmath.org/0516.10036}{Grandes valeurs de la fonction somme des diviseurs et hypothèse de {R}iemann}},
  author={Robin, G},
  year={1984},
  journal={Journal de Mathématiques Pures et Appliquées},
  volume={63},
  number={2},
  pages={187–213}
}

@article{choie2007robin,
title={\href{https://doi.org/10.5802/jtnb.591}{On {R}obin's criterion for the {R}iemann hypothesis}},
  author={Choie, YoungJu and Lichiardopol, Nicolas and Moree, Pieter and Sol{\'e}, Patrick},
  journal={Journal de th{\'e}orie des nombres de Bordeaux},
  volume={19},
  number={2},
  pages={357--372},
  year={2007}
}

@book{broughan2017equivalents,
  title={\href{https://doi.org/10.1017/9781108178228}{Equivalents of the {R}iemann Hypothesis: Volume 1, Arithmetic Equivalents}},
  author={Broughan, Kevin},
  volume={164},
  year={2017},
  publisher={Cambridge University Press}
}
\end{document}